\pdfoutput=1
\RequirePackage{ifpdf}
\ifpdf 
\documentclass[pdftex]{sigma}
\else
\documentclass{sigma}
\fi

\usepackage{bbm}
\usepackage{mathrsfs}
\usepackage{enumerate}
\usepackage{tikz}
\usetikzlibrary{arrows,fit,shapes.geometric}
\tikzset{mynode/.style={
 ellipse,draw,
 inner xsep=-1pt, color=red
 }}
\usepackage{tikz-cd}
\numberwithin{equation}{section}
\newtheorem{Theorem}{Theorem}[section]
\newtheorem*{Theorem*}{Theorem}
\newtheorem{Lemma}[Theorem]{Lemma}
\newtheorem{Corollary}[Theorem]{Corollary}
\newtheorem{Proposition}[Theorem]{Proposition}

{\theoremstyle{definition}
\newtheorem{Definition}[Theorem]{Definition}
\newtheorem{Remark}[Theorem]{Remark}

\newtheorem*{Example*}{Example}
\newtheorem{Notation}[Theorem]{Notation}

}

\newcommand{\Cset}{\mathbb{C}}

\newcommand{\Nset}{\mathbb{N}}
\newcommand{\Rset}{\mathbb{R}}

\newcommand{\Zset}{\mathbb{Z}}
\newcommand{\cA}{\mathcal{A}}
\newcommand{\cB}{\mathcal{B}}
\newcommand{\cC}{\mathcal{C}}
\newcommand{\cD}{\mathcal{D}}

\newcommand{\cI}{\mathcal{I}}

\newcommand{\cL}{\mathcal{L}}
\newcommand{\cM}{\mathcal{M}}
\newcommand{\cN}{\mathcal{N}}

\newcommand{\tvarphi}{\tilde{\varphi}}

\newcommand{\calM}{\mathcal{M}}
\newcommand{\tcM}{\widetilde{\mathcal{M}}}
\newcommand{\tpsi}{\widetilde{\psi}}
\newcommand{\talpha}{\widetilde{\alpha}}
\newcommand{\tbeta}{\widetilde{\beta}}
\newcommand{\tgamma}{\widetilde{\gamma}}
\newcommand{\tiota}{\widetilde{\iota}}
\newcommand{\trho}{\widetilde{\rho}}
\newcommand{\hcM}{\widehat{\mathcal{M}}}
\newcommand{\hpsi}{\widehat{\psi}}
\newcommand{\halpha}{\widehat{\alpha}}
\newcommand{\hbeta}{\widehat{\beta}}
\newcommand{\hgamma}{\widehat{\gamma}}
\newcommand{\hiota}{\widehat{\iota}}

\newcommand{\1}{\mathbbm{1}}

\newcommand{\Aut}{\operatorname{Aut}}
\newcommand{\End}{\operatorname{End}}

\newcommand{\trace}{\operatorname{tr}}
\newcommand{\gen}{\operatorname{gen}}
\newcommand{\sh}{\operatorname{sh}}
\newcommand{\Id}{\operatorname{Id}}

\begin{document}
\allowdisplaybreaks

\newcommand{\arXivNumber}{2204.03595}

\renewcommand{\thefootnote}{}

\renewcommand{\PaperNumber}{083}

\FirstPageHeading

\ShortArticleName{Markovianity and the Thompson Group $F$}

\ArticleName{Markovianity and the Thompson Group $\boldsymbol{F}$\footnote{This paper is a~contribution to the Special Issue on Non-Commutative Algebra, Probability and Analysis in Action. The~full collection is available at \href{https://www.emis.de/journals/SIGMA/non-commutative-probability.html}{https://www.emis.de/journals/SIGMA/non-commutative-probability.html}}}

\Author{Claus K\"OSTLER~$^{\rm a}$ and Arundhathi KRISHNAN~$^{\rm b}$}\vspace{1.3mm}

\AuthorNameForHeading{C.~K\"ostler and A.~Krishnan}

\Address{$^{\rm a)}$~School of Mathematical Sciences, University College Cork, Cork, Ireland}
\EmailD{\href{mailto:claus@ucc.ie}{claus@ucc.ie}}\vspace{1.3mm}

\Address{$^{\rm b)}$~Department of Pure Mathematics, University of Waterloo, Ontario, Canada}
\EmailD{\href{mailto:arundhathi.krishnan@uwaterloo.ca}{arundhathi.krishnan@uwaterloo.ca}}\vspace{1.3mm}

\ArticleDates{Received April 08, 2022, in final form October 07, 2022; Published online October 27, 2022}\vspace{1.3mm}

\Abstract{We show that representations of the Thompson group $F$ in the automorphisms of a noncommutative probability space yield a large class of bilateral stationary noncommutative Markov processes. As a partial converse, bilateral stationary Markov processes in tensor dilation form yield representations of $F$. As an application, and building on a~result of K\"ummerer, we canonically associate a representation of $F$ to a bilateral stationary Markov process in classical probability.\vspace{1.3mm}}

\vspace{1.3mm}

\Classification{46L53; 60J05; 60G09; 20M30}\vspace{2mm}

\renewcommand{\thefootnote}{\arabic{footnote}}
\setcounter{footnote}{0}

\section{Introduction}\vspace{2mm}

The Thompson group $F$ was introduced
by Richard Thompson in the 1960s and many of its unusual,
interesting properties
\cite{CF11, CFP96} have been deeply studied over the past
decades, in particular due to the still open conjecture of its
nonamenability. Recently Vaughan Jones provided a new approach to the construction of (unitary) representations of the
Thompson group $F$ which is motivated by the link between
subfactor theory and conformal field theory (see \cite{AJ21,BJ19b,BJ19a,Jo17,Jo18a,Jo18b}). Independently,
another approach to the representation theory of the Thompson
group $F$ is motivated
by recent progress in the study of distributional invariance
principles and symmetries in noncommutative probability (see \cite{EGK17,Ko10} and \cite[Introduction]{KKW20}). More
precisely, a close relation between certain representations of
the Thompson monoid $F^+$ and unilateral noncommutative
stationary Markov processes is established in \cite{KKW20}.
The goal of the present paper is to demonstrate that this
connection appropriately extends to one between
representations of the Thompson group $F$ and bilateral
stationary noncommutative Markov processes (in the sense of K\"ummerer \cite{Ku85}).
Throughout we will mainly focus on a conceptual framework that is relevant in the operator algebraic reformulation of stationary Markov processes in classical probability theory.\looseness=1

One of our main results is Theorem \ref{theorem:markov-filtration-1}
which is about the construction of a local Markov filtration
and a bilateral stationary Markov process from a given representation of the Thompson group $F$.
Going beyond the framework of Markovianity, this construction is further deepened in Theorem \ref{theorem:markov-filtration-2} and Corollary \ref{corollary:triangulararray},
to obtain rich triangular arrays of commuting squares.
A~main result in the converse direction is Theorem \ref{theorem:TensorMarkovF}, where we provide a canonical construction of a representation of the Thompson group $F$ from a given bilateral stationary noncommutative Markov process in tensor dilation form. Finally, we apply this canonical construction to bilateral stationary Markov processes in classical probability. We establish in Theorem \ref{theorem:F-gen-compression} that, for a given Markov transition operator, there exists a representation of the Thompson group $F$ such that this Markov transition operator
is the compression of a represented generator of the Thompson group $F$.

We keep the presentation of our results on the connection between representations of the Thompson group $F$ and Markovianity as close as possible to our treatment for the Thompson monoid $F^+$ in \cite{KKW20}. Here we focus on the dynamical systems approach for noncommutative stationary processes and deliberately omit reformulations in terms of noncommutative random variables. In parts this is attributed to the fact that usually
the noncommutative probability space generated by a bilateral stationary Markov sequence of noncommutative random variables turns out to be ``too small'' to accommodate
a representation of the Thompson group $F$. This is in contrast to the situation in \cite{KKW20}, where unilateral stationary
Markov sequences generate a noncommutative probability space which is large enough to support a representation of the Thompson monoid $F^+$. Some of these conceptual differences are further discussed and illustrated in the closing Section~\ref{subsection:discuss-classical}.
Therein we constrain ourselves to the basics of the construction of representations of the Thompson group $F$ from a given Markov transition operator and postpone a more-in-depth structural discussion to the future.

Let us outline the content of this paper. Section~\ref{section:Preliminaries} starts with providing
definitions, notation and some background results on the
Thompson group $F$ (see Section~\ref{subsection:basics-on-F}). The basics of noncommutative
probability spaces and Markov maps are given in Section~\ref{subsection:Markov-maps}.
We review in Section~\ref{subsection:Markovianity} the notion of
commuting squares from subfactor theory, as it underlies the present concept of Markovianity in noncommutative probability. Furthermore, we provide the notion of a local Markov filtration which allows us to define Markovianity
on the level of von Neumann subalgebras without any reference to noncommutative random variables. Finally, we review some results on noncommutative stationary processes in Section~\ref{subsection:Noncommutative Stationary Processes}. Here we will meet bilateral noncommutative stationary Markov processes and Markov dilations in the sense of K\"ummerer~\cite{Ku85} as well as bilateral noncommutative stationary Bernoulli shifts.

We investigate in Section~\ref{section:Mark-From-Rep} how representations of the Thompson group
$F$ in the automorphisms of noncommutative probability spaces yield bilateral noncommutative stationary Markov processes. Section~\ref{subsection:generating} introduces the generating property of representations of $F$ in
Definition~\ref{definition:generating}. This property ensures that the fixed point algebras of the
represented generators of $F$ form a tower which generates the noncommutative probability
space, see Proposition~\ref{proposition:generating-property}. This tower of fixed
point algebras equips the noncommutative probability space with a filtration which, using actions
of the represented generators, can be further upgraded to become a local Markov filtration.
Section~\ref{subsection:Markov-F} considers certain noncommutative stationary processes which are adapted to this local Markov filtration.

The closing Section \ref{section:Reps-of-F-from-Mark} shows that representations of $F$ can be obtained from an important class of bilateral stationary noncommutative Markov processes. To be more precise, in Section~\ref{subsection:Example} we provide elementary constructions of the Thompson group $F$ in the automorphisms of a tensor product von Neumann algebra. This extends the representation of the Thompson monoid $F^+$ obtained in \cite{KKW20} and also provides examples of bilateral noncommutative Markov and Bernoulli shifts. We show in Section~\ref{subsection:constr-rep-F} that Markov processes in tensor dilation form give rise to representations of $F$. Finally, in Section~\ref{subsection:constr-classical} we use a result of K\"ummerer to show that, given a bilateral stationary Markov process
in the classical case, we can obtain representations of $F$ such that the associated transition operator is the compression of a represented generator of $F$. We provide more details to further motivate the construction of these representations in Section~\ref{subsection:discuss-classical}, also pointing out differences between the unilateral and bilateral cases in the process.\looseness=1

\section{Preliminaries} \label{section:Preliminaries}
\subsection[The Thompson group F]{The Thompson group $\boldsymbol{F}$}
\label{subsection:basics-on-F}

The Thompson group $F$, originally introduced by Richard Thompson in 1965 as a certain group of piece-wise linear homeomorphisms on the interval $[0,1]$, is known to have the infinite
presentation
\begin{equation*}
F:=\langle g_0,g_1,g_2,\ldots \mid g_{k}g_{\ell}=g_{\ell+1}g_{k}
\text{ for } 0\leq k<\ell <\infty \rangle.
\end{equation*}
We note that we work throughout with
generators $g_k$ which correspond to the
inverses of the generators usually used in the literature (e.g.,~\cite{Be04}).
Let $e \in F$ denote the neutral element.
As it is well-known, $F$ is finitely generated with $F =
\langle g_0, g_1 \rangle$. Furthermore,
as shown for example in \cite[Theorem 1.3.7]{Be04}, an element $e \neq g \in F$ has the unique normal form
\begin{align} \label{eq:F-normal-form}
g = g_0^{-b_0}\cdots g_k^{-b_k} g_k^{a_k}\cdots g_0^{a_0},
\end{align}
where $a_0, \ldots , a_k , b_0, \ldots , b_k \in \Nset_0$, $k \geq 0$ and
 \begin{enumerate}\itemsep=0pt
 \item[($i$)] exactly one of $a_k$ and $b_k$ is non-zero,
 \item[($ii$)] if $a_i\neq 0$ and $b_i\neq 0$, then $a_{i+1}\neq 0$
or $b_{i+1}\neq 0$.
 \end{enumerate}
As the defining relations of this presentation of $F$ involve no inverse generators, one can associate
to it the monoid
\begin{equation}\label{eq:F+}
F^{+}=\langle g_0,g_1,g_2,\ldots \mid g_{k}g_{\ell}=g_{\ell+1}g_{k}
\text{ for } 0\leq k<\ell <\infty \rangle^+,
\end{equation}
referred to as the \emph{Thompson monoid $F^+$}. We remark that, alternatively, the
generators of this monoid can be obtained as morphisms (in the inductive limit) of the category
of finite binary forests, see for example \cite{Be04,Jo18a}.

\begin{Definition} \label{definition:mn-shift}
Let $m,n \in \Nset_0$ with $m \le n$ be fixed. The \emph{$(m,n)$-partial shift} $\sh_{m,n}$ is
the group homomorphism on $F$ defined by
\[
\sh_{m,n}(g_k) = \begin{cases}
 g_m &\text{if}\ k=0,
 \\
 g_{n +k} &\text{if}\ k \ge 1.
 \end{cases}
\]
\end{Definition}

We remark that the map $\sh_{m,n}$ preserves all defining relations of $F$ and is thus
well-defined as a group homomorphism.

\begin{Lemma} 
The group homomorphisms $\sh_{m,n}$ on $F$ are injective for all $m,n \in \Nset_0$.
\end{Lemma}

\begin{proof}
It suffices to show that $\sh_{m,n}(g) = e$ implies $g=e$.
Let $g\in F$ have the (unique) normal form as stated in \eqref{eq:F-normal-form}. Thus, by the definition of the partial shifts,
\[
\sh_{m,n}(g) = g_m^{-b_0}\cdots g_{n+k}^{-b_k} \, g_{n+k}^{a_k}\cdots g_m^{a_0}.
\]
 Thus $\sh_{m,n}(g) = e$ if and only if
 $g_{n+k}^{a_k}\cdots \, g_m^{a_0} = g_{n+k}^{b_k} \cdots \, g_m^{b_0}$.
Since the elements on both sides of the last equation are in normal form, its uniqueness implies $a_i =b_i$ for all $i$. But this entails $g=e$.
\end{proof}

\subsection{Noncommutative probability spaces and Markov maps}
\label{subsection:Markov-maps}

Throughout, a \emph{noncommutative probability space} $(\cM,\psi)$ consists of a
von Neumann algebra $\cM$ and a faithful normal state $\psi$ on $\cM$. The identity
of $\cM$ will be denoted by $\1_{\cM}$, or simply by $\1$ when the context is clear.
Throughout, $\bigvee_{i\in I}\cM_i$ denotes the von Neumann algebra generated by the
family of von Neumann algebras $\{\cM_i\}_{i\in I} \subset \cM$ for $I\subset \Zset$.
If $\cM$ is abelian and acts on a~separable Hilbert space, then $(\cM,\psi)$ is
isomorphic to $\big(L^\infty(\Omega, \Sigma, \mu), \int_{\Omega} \cdot \,\, {\rm d}\mu\big)$
for some standard probability space $(\Omega, \Sigma, \mu)$.

\begin{Definition} 
An \emph{endomorphism} $\alpha$ of a noncommutative probability space $(\cM,\psi)$ is a $*$-homomorphism on $\cM$
satisfying the following additional properties:
\begin{enumerate}\itemsep=0pt
 \item[($i$)] $\psi\circ \alpha=\psi$ (stationarity),
 \item[($ii$)] $\alpha$ and the modular automorphism group $\sigma_t^{\psi}$ commute
 for all $t\in \Rset$ (modularity).
\end{enumerate}
The set of endomorphisms of $(\cM,\psi)$ is denoted by $\End(\cM,\psi)$. We note that an
endomorphism of $(\cM,\psi)$ is automatically injective. In this paper, we will chiefly work with the automorphisms of $(\cM,\psi)$ denoted by $\Aut(\cM,\psi)$.
\end{Definition}
Note that $\alpha \in \End(\cM,\psi)$ automatically satisfies
\[
\alpha(\1_{\cM})=\1_{\cM} \qquad \text{(unitality)}.
\]
Indeed, the *-homomorphism property and stationarity of $\alpha$ entails
\[
\psi\big( (\alpha(\1_{\cM}) -\1_{\cM})^* (\alpha(\1_{\cM}) -\1_{\cM})\big) = 0.
\]
Now the faithfulness of $\psi$ ensures
$\alpha(\1_{\cM}) -\1_{\cM} = 0$.

\begin{Definition}
Let $(\cM,\psi)$ and $(\cN,\varphi)$ be two noncommutative probability spaces. A linear map
$T \colon \cM \to \cN$ is called a \emph{$(\psi,\varphi)$-Markov map} if the following
conditions are satisfied:
\begin{enumerate}\itemsep=0pt
\item[$(i)$] 
$T$ is completely positive,
\item[$(ii)$] 
$T$ is unital,
\item[$(iii)$]
$\varphi \circ T = \psi$,
\item[$(iv)$] 
$T \circ \sigma_t^{\psi} = \sigma_t^{\varphi} \circ T$, for all $t \in \Rset$.
\end{enumerate}
\end{Definition}

Here $\sigma_{}^{\psi}$ and $\sigma_{}^{\varphi}$ denote the modular automorphism groups of
$(\cM,\psi)$ and $(\cN,\varphi)$, respectively. If $(\cM,\psi) = (\cN,\varphi)$, we say that
$T$ is a $\psi$-\emph{Markov map on $\cM$}. Conditions~$(i)$ to~$(iii)$
imply that a Markov map is automatically normal. The condition~$(iv)$ is equivalent
to the condition that a unique Markov map $T^* \colon (\cN,\varphi) \to (\cM,\psi)$ exists such
that
\[
\psi\big(T^*(y)x\big) = \varphi\big(y\, T(x)\big), \qquad x \in \cM,\quad y \in \cN.
\]
The Markov map $T^*$ is called the \emph{adjoint} of $T$ and $T$ is called \emph{self-adjoint} if
$T=T^*$. We note that condition~$(iv)$ is automatically satisfied whenever $\psi$ and
$\varphi$ are tracial, in particular for abelian von Neumann algebras $\cM$ and $\cN$.
Furthermore, we note that any
$T \in \End(\cM,\psi)$ is automatically a $\psi$-Markov map and, in particular, any $T \in \Aut(\cM,\psi)$ is a $\psi$-Markov map with adjoint $T^* = T^{-1}$.

We recall for the convenience of the reader the definition of conditional expectations in the
present framework of noncommutative probability spaces.
\begin{Definition}
Let $(\cM,\psi)$ be a noncommutative probability space, and $\cN$ be a von Neumann subalgebra of
$\cM$. A linear map $E\colon \cM\to \cN$ is called a \emph{conditional expectation} if it satisfies
the following conditions:
\begin{enumerate}\itemsep=0pt
 \item[$(i)$] $E(x)=x$ for all $x\in \cN$,
 \item[$(ii)$] $\|E(x)\|\leq \|x\|$ for all $x\in \cM$,
 \item[$(iii)$] $\psi\circ E=\psi$.
\end{enumerate}
\end{Definition}

Such a conditional expectation exists if and only if $\cN$ is globally invariant under the
modular automorphism group of $(\cM,\psi)$ (see \cite{Ta72,Ta79,Ta03}). The
von Neumann subalgebra $\cN$ is called $\psi$-conditioned if this condition is satisfied. Note
that such a conditional expectation is automatically normal and uniquely determined by $\psi$. In
particular, a conditional expectation is a~Markov map and satisfies the module property
$E(axb)=aE(x)b$ for $a,b\in \cN$ and $x\in \cM$.

\subsection{Noncommutative independence and Markovianity}\label{subsection:Markovianity}

We recall some equivalent properties as they serve to define commuting squares in subfactor
theory (see for example \cite{GHJ89,JS97,Po89}) and as they are familiar from conditional independence
in classical probability.

\begin{Proposition}\label{proposition:cs}
Let $\cM_0$, $\cM_1$, $\cM_2$ be $\psi$-conditioned von Neumann subalgebras of the probability space
$(\cM,\psi)$ such that $\cM_0 \subset (\cM_1 \cap \cM_2)$. Then the following are equivalent:
\begin{enumerate}\itemsep=0pt
\item[$(i)$] \label{item:cs-i}
$E_{\cM_0}(xy) = E_{\cM_0}(x) E_{\cM_0}(y)$ for all $x \in \cM_1$ and $y\in \cM_2$,
\item[$(ii)$] \label{item:cs-ii}
$E_{\cM_1} E_{\cM_2} = E_{\cM_0}$,
\item[$(iii)$] \label{item:cs-iii}
$E_{\cM_1}(\cM_2) = \cM_0$,
\item[$(iv)$] \label{item:cs-iv}
$E_{\cM_1} E_{\cM_2} = E_{\cM_2} E_{\cM_1}$ and $\cM_1\cap \cM_2 = \cM_0$.
\end{enumerate}
In particular, it holds that $\cM_0 = \cM_1 \cap \cM_2$ if one and thus all of these four
assertions are satisfied.
\end{Proposition}

\begin{proof}
The case of tracial $\psi$ is proved in \cite[Proposition~4.2.1]{GHJ89}. The non-tracial case follows
from this, after some minor modifications of the arguments therein.
\end{proof}

\begin{Definition}
The inclusions
\[
\begin{matrix}
\cM_2 &\subset &\cM\\
\cup & & \cup \\
\cM_0 & \subset & \cM_1
\end{matrix}
\]
as given in Proposition~\ref{proposition:cs} are said to form a \emph{commuting square} (\emph {of von
Neumann algebras}) if one (and thus all) of the equivalent conditions~$(i)$ to~$(iv)$
are satisfied in Proposition~\ref{proposition:cs}.
\end{Definition}

\begin{Notation} 
We write $I < J$ for two subsets $I, J \subset \Zset$ if $i < j$ for all $i \in I$ and ${j \in
J}$. The cardinality of $I$ is denoted by $|I|$. For $N \in \Zset$, we denote by $I + N$ the
shifted set $\{i + N \mid i \in I\}$. Finally, $\cI(\Zset)$ denotes the set of all ``intervals'' of
$\Zset$, i.e.,~sets of the form $[m,n] := \{m, m+1, \ldots, n\}$, $[m,\infty) := \{m, m+1,
\ldots\}$ or $(-\infty, m] := \{\ldots, m-1,m\}$ for $-\infty < m \le n < \infty$.
\end{Notation}

We next address the basic notions of Markovianity in noncommutative probability. Commonly,
Markovianity is understood as a property of random variables relative to a filtration of the
underlying probability space. Our investigations from the viewpoint of distributional invariance
principles reveal that the phenomenon of ``Markovianity'' emerges without reference to any
stochastic process already on the level of a family of von Neumann subalgebras, indexed by the
partially ordered set of all ``intervals'' $\cI(\Zset)$. As commonly the index set of a filtration
is understood to be totally ordered \cite{Ve17}, we refer to such families
with partially ordered index sets as ``local filtrations''.

\begin{Definition}
A family of $\psi$-conditioned von Neumann subalgebras $\cM_\bullet \equiv \{\cM_I\}_{I \in
\cI(\Zset)}$ of the probability space $(\cM,\psi)$ is called a \emph{local filtration $($of
$(\cM,\psi))$} if
 \begin{gather*}
 I \subset J \, \Longrightarrow \, \cM_I \subset \cM_J \qquad \text{(isotony)}.
\end{gather*}
\end{Definition}

The isotony property ensures that one has the inclusions
\[
\begin{matrix}
\cM_{I} &\subset &\cM\\
\cup & & \cup \\
\cM_{K} & \subset & \cM_{J}
\end{matrix}
\]
for $I, J, K \in \cI(\Zset)$ with $K \subset (I \cap J)$. Finally, let $\cN_\bullet \equiv
\{\cN_I\}_{I \in \cI(\Zset)}$ be another local filtration of $(\cM,\psi)$. Then $\cN_\bullet$
is said to be \emph{coarser} than $\cM_\bullet$ if $\cN_I \subset \cM_I$ for all
$I \in \cI(\Zset)$ and we denote this by $\cN_{\bullet}\prec \cM_{\bullet}$. Occasionally we
will address $\cN_{\bullet}$ also as a \emph{local subfiltration} of $\cM_{\bullet}$.

\begin{Definition} 
Let $\cM_\bullet \equiv \{\cM_I\}_{I \in \cI(\Zset)}$ be a local filtration of $(\cM,\psi)$.
$\cM_{\bullet}$ is said to be \emph{Markovian} if the inclusions
\[
\begin{matrix}
\cM_{(-\infty,n]} &\subset &\cM\\
\cup & & \cup \\
\cM_{[n,n]} & \subset & \cM_{[n,\infty)}
\end{matrix}
\]
form a commuting square for each $n \in \Zset$.
\end{Definition}

Cast as commuting squares, Markovianity of the local filtration $\cM_\bullet$ has many equivalent
formulations, see Proposition~\ref{proposition:cs}. In particular, it holds that
\begin{gather*}
E_{\cM_{(-\infty,n]}} E_{\cM_{[n,\infty)}} = E_{\cM_{[n,n]}}\qquad \text{for all}\quad n \in \Zset.
\end{gather*}
Here $E_{\cM_I}$ denotes the $\psi$-preserving normal conditional expectation from $\cM$ onto
$\cM_I$.

\subsection{Noncommutative stationary processes and dilations}
\label{subsection:Noncommutative Stationary Processes}

We introduce bilateral noncommutative stationary processes, as they underlie the
approach to distributional invariance principles in \cite{GK09,Ko10}. Furthermore, we present
 dilations of Markov maps using K\"ummerer's approach to
noncommutative stationary Markov processes \cite{Ku85}. The existence of such dilations is
actually equivalent to the factoralizability of Markov maps (see \cite{AD06,HM11}).

\begin{Definition} \label{definition:process-sequence}
A \emph{bilateral stationary process} $(\cM,\psi, \alpha, \cA_0)$ consists of a probability
space $(\cM,\psi)$, a $\psi$-conditioned subalgebra $\cA_0 \subset \cM$, and an automorphism
$\alpha\in \Aut(\cM,\psi)$. The sequence
\[
(\iota_n)_{n \in \Zset}\colon\ (\cA_0, \psi_0) \to (\cM,\psi),
\qquad \iota_{n} := \alpha^n|_{\cA_0}=\alpha^n\iota_0,
\]
is called the \emph{sequence of random variables associated to} $(\cM,\psi, \alpha, \cA_0)$.
Here $\psi_0$ denotes the restriction of $\psi$ from $\cM$ to $\cA_0$ and $\iota_0$ denotes the
inclusion map of $\cA_0$ in $\cM$.

The stationary process $(\cM,\psi, \alpha, \cA_0)$ is called \emph{minimal} if
\[
\bigvee_{i \in \Zset} \alpha^i(\cA_0) = \cM.
\]
\end{Definition}

\begin{Definition} \label{definition:ncms}
The (not necessarily minimal) stationary process $(\cM,\psi, \alpha, \cA_0)$ is called a~(\emph{bilateral noncommutative}) \emph {stationary Markov process} if its canonical local filtration
\[
\biggl\{\cA_I:= \bigvee_{i \in I} \alpha^i(\cA_0)\biggr\}_{I \in \cI(\Zset)}
\]
is Markovian. If this process is minimal, then the endomorphism $\alpha$ is also called a
\emph{Markov shift} with generator $\cA_0$. Furthermore, the associated $\psi_0$-Markov map $T=\iota_0^*\alpha \iota_0$ on $\cA_0$ is called the
\emph{transition operator} of the stationary Markov process.
Here $\iota_0$ denotes the inclusion
map of $\cA_0$ in $\cM$, and $\psi_0$ is the restriction of $\psi$ to $\cA_0$.
\end{Definition}

The next lemma gives a simplified condition to check that a bilateral stationary process is a~Markov process.

\begin{Lemma} \label{lemma:Mark-Suff}
Let $(\cM,\psi, \alpha, \cA_0)$ be a bilateral stationary process with canonical local filtration
$
\{\cA_I:= \bigvee_{i \in I} \alpha^i(\cA_0)\}_{I \in \cI(\Zset)}$.
Suppose
\[
P_{(-\infty, 0]} P_{[0,\infty)} = P_{[0,0]},
\]
where $P_I$ denotes the $\psi$-preserving normal conditional expectation from $\cM$ onto $\cA_I$.
Then $\{\cA_I\}_{I \in \cI(\Zset)}$ is a local Markov filtration and $( \cM,\psi, \alpha, \cA_0)$ is a bilateral stationary Markov process.
\end{Lemma}

\begin{proof}
For all $k \in \Zset$ and $I \in \cI(\Zset)$, we have $ \alpha^{k} P_{I} = P_{I+k} \alpha^{k}$ (see \cite[Remark 2.1.4]{Ku85}). Hence, for each $n \in \Zset$,
\begin{align*}
P_{(-\infty,0]}P_{[0,\infty)} = P_{[0,0]}
&\,\Longleftrightarrow\,
\alpha^n P_{(-\infty,0]}P_{[0,\infty)}\alpha^{-n}
 = \alpha^n P_{[0,0]} \alpha^{-n}
\\
&\,\Longleftrightarrow\,
P_{(-\infty,n]}P_{[n,\infty)} = P_{[n,n]},
\end{align*}
which is the required
Markovianity for the local filtration $\{\cA_{I}\}_{I\in \cI(\Zset)}$.
\end{proof}

\begin{Definition}[{\cite[Definition 2.1.1]{Ku85}}] \label{definition:dilation}
Let $(\cA,\varphi)$ be a probability space. A $\varphi$-Markov map~$T$ on $\cA$ is said to admit
a \emph{$($bilateral state-preserving$)$ dilation} if there exists a probability space $(\cM,\psi)$,
an automorphism $\alpha\in \Aut(\cM,\psi)$ and a $(\varphi,\psi)$-Markov map $\iota_0\colon\cA\to \cM$
such that
\begin{eqnarray*}
T^n=\iota_0^*\alpha^n \iota_0 \qquad \text{for all}\quad n \in \Nset_0.
\end{eqnarray*}
Such a dilation of $T$ is denoted by the quadruple $(\cM,\psi,\alpha,\iota_0)$ and is said to be
\emph{minimal} if $\cM=\bigvee_{n\in \Zset} \alpha^{n}\iota_0(\cA)$. $(\cM,\psi,\alpha,\iota_0)$ is called a \emph{dilation of first order} if the equality $T= \iota_0^* \alpha \iota_0$ alone holds.
\end{Definition}

Actually it follows from the case $n=0$ that the $(\varphi,\psi)$-Markov map $\iota_0$ is a
random variable from $(\cA,\varphi)$ to $(\cM,\psi)$ such that $\iota_0\iota_0^*$ is the
$\psi$-preserving conditional expectation from $\cM$ onto~$\iota_0(A)$.

\begin{Definition}[{\cite[Definition 2.2.4]{Ku85}}] 
The dilation $(\cM,\psi,\alpha,\iota_0)$ of the $\varphi$-Markov map $T$ on~$\cA$ (as introduced
in Definition \ref{definition:dilation}) is said to be a (bilateral state-preserving) \emph {Markov
dilation} if the local filtration $\big\{\cA_I := \bigvee_{n \in I} \alpha^n\iota_0(\cA)\big\}_{I \in
\cI(\Zset)}$ is Markovian.
\end{Definition}

\begin{Remark}
A dilation of a $\varphi$-Markov map $T$ on $\cA$ may not be a Markov dilation.
This is discussed in \cite[Section 3]{KuSchr83}, where it is shown that Varilly has
constructed a dilation in \cite{Va81} which is not a Markov dilation. We are grateful
to B.~K\"ummerer for bringing this to our attention \cite{Ku21}.
Note that this does not contradict the result that the \emph{existence} of a dilation and
the \emph{existence} of a Markov dilation are equivalent (see \cite[Theorem 4.4]{HM11} or \cite[Theorem 2.6.8]{KKW20}).
\end{Remark}

\begin{Definition}[{\cite[Definition 4.1.3]{Ku85}}] \label{definition:tensordilation}
Let $(\cA, \varphi)$ be a probability space and $T$ be a $\varphi$-Markov map on $\cA$. A dilation of first order $(\cM, \psi, \alpha, \iota_0)$ of $T$ is called a \emph{tensor dilation} if the conditional expectation $\iota_0 \iota_0^* \colon \cM \to \iota_0(\cA)$ is of tensor type, that is, there exists a von Neumann subalgebra~$\cC$ of $\cM$ with faithful normal state $\chi$ such that $\mathcal{M} = \mathcal{A} \otimes \mathcal{C}$ and $(\iota_0\iota_0^*)(a \otimes x) = \chi(x)(a \otimes \mathbbm{1}_{\mathcal{C}})$ for all $a \in \cA$, $x \in \cC$.
\end{Definition}
Let us next relate the above bilateral notions of dilations and stationary processes. It is
immediate that a dilation $(\cM,\psi,\alpha,\iota_0)$ of the $\varphi$-Markov map $T$
on $\cA$ gives rise to the stationary process $(\cM,\psi, \alpha, \iota_0(\cA))$.
Furthermore, this stationary process is Markovian if and only if the dilation is a Markov
dilation, as evident from the definitions. Conversely, a stationary Markov process
yields a dilation (and thus a Markov dilation) as it was shown by K\"ummerer, stated below for the convenience of the reader.

 \begin{Proposition}[{\cite[Proposition~2.2.7]{Ku85}}]
 \label{proposition:dilation}
Let $(\cM,\psi, \alpha,\cA_0)$ be a bilateral noncommutative stationary Markov process and $T=\iota_0^*\alpha \iota_0$ be
the corresponding transition operator where $\iota_0$ is the inclusion map of $\cA_0$ into $\cM$.
Then $(\cM,\psi,\alpha,\iota_0)$ is a dilation of $T$. In other words, the following diagram
commutes for all $n\in \Nset_0$:
	\[
	\begin{tikzcd}
	(\cA_0, \psi_0) \arrow[r, "T^n"] \arrow[d, "\iota_0"]
	& (\cA_0, \psi_0) \arrow[d, leftarrow, "\iota_0^*"] \\
	(\cM,\psi) \arrow[r, "\alpha^n"]
	& (\cM,\psi).
	\end{tikzcd}
	\]
Here $\psi_0$ denotes the restriction of $\psi$ to $\cA_0$.
\end{Proposition}

We close this section by providing a noncommutative notion of operator-valued Bernoulli
shifts. The definition of such shifts stems from
investigations of K\"ummerer on the structure of noncommutative Markov processes in \cite{Ku85},
and such shifts can also be seen to emerge from the noncommutative extended de Finetti theorem in
\cite{Ko10}.

In the following, $\cM^\beta := \{x \in \cM \mid \beta(x) = x\}$ denotes the fixed point algebra of $\beta \in \Aut(\cM,\psi)$.
Note that $\cM^\beta$ is automatically a $\psi$-conditioned von Neumann subalgebra.

\begin{Definition} 
The minimal stationary process $(\cM,\psi, \beta, \cB_0)$
with canonical local filtration $\big\{\cB_I = \bigvee_{i \in I} \beta_0^i(\cB_0)\big\}_{I \in \cI(\Zset)}$
is called a \emph{bilateral
noncommutative Bernoulli shift} with \emph{generator}~$\cB_0$ if $\cM^{\beta}
\subset \cB_0$ and
\[
	\begin{matrix}
	\cB_{I} &\subset &\cM\\
	\cup & & \cup \\
	\cM^{\beta} & \subset & \cB_{J}
	\end{matrix}
\]
forms a commuting square for any $I, J \in \cI(\Zset)$ with $I \cap J = \varnothing$.
\end{Definition}

It is easy to see that a noncommutative Bernoulli shift $( \cM,\psi, \beta, \cB_0)$ is a minimal
stationary Markov process where the corresponding transition operator $\iota_0^*\beta \iota_0$
is a conditional expectation (onto $\cM^{\beta}$, the fixed point algebra of $\beta$).
Here $\iota_0$ denotes the inclusion map of $\cB_0$ into $\cM$.

\section[Markovianity from representations of F]{Markovianity from representations of $\boldsymbol F$} \label{section:Mark-From-Rep}

We show that bilateral stationary Markov processes can be obtained from representations of the Thompson group $F$ in the automorphisms of a noncommutative probability space. Most of the results in this section follow closely those of \cite[Section 4]{KKW20}, suitably adapted to the bilateral case.

Let us fix some notation, as it will be used throughout this section. We assume that the
probability space $(\cM,\psi)$ is equipped with the representation $\rho \colon F \to
\Aut(\cM,\psi)$. For brevity of notion, especially in proofs, the represented generators of $F$
are also denoted by
\[
\alpha_n := \rho(g_n) \in \Aut(\cM,\psi),
\]
with fixed point algebras given by $\cM^{\alpha_n} := \{x \in \cM \mid \alpha_n (x) = x\}$,
for $0 \le n < \infty$. Of course, $\cM^{\alpha_n} = \cM^{\alpha_n^{-1}}$. Furthermore, the intersections of fixed point algebras
\[
\cM_n := \bigcap_{k \ge n +1} \cM^{\alpha_k}
\]
give the tower of von Neumann subalgebras
\[
\cM^{\rho(F)} \subset \cM_0 \subset \cM_1 \subset \cM_2 \subset \cdots \subset \cM_\infty := \bigvee_{n \ge 0} \cM_n \subset \cM.
\]
From the viewpoint of noncommutative probability theory, this tower provides a filtration of the
noncommutative probability space $(\cM,\psi)$. The canonical local filtration of
a stationary process $(\cM, \psi, \alpha_0, \cA_0)$ will be seen to be a local subfiltration of a local
Markov filtration whenever the $\psi$-conditioned von Neumann subalgebra $\cA_0$ is
well-localized, to be more precise: contained in the intersection of fixed point algebras
$\cM_0$. It is worthwhile to emphasize that, depending on the choice of the generator $\cA_0$,
the canonical local filtration of this stationary process may not be Markovian. Section~\ref{subsection:Markov-F} investigates in detail conditions under which the canonical local filtration
of a stationary process $(\cM, \psi, \alpha_0, \cA_0)$ is Markovian. 
\subsection{Representations with a generating property}
\label{subsection:generating}
An immediate consequence of the relations between generators of the Thompson group $F$ is
the adaptedness of the endomorphism $\alpha_0$ to the tower of (intersected) fixed point
algebras:
\[
\alpha_0(\cM_{n}) \subset \cM_{n+1} \qquad \text{for all}\quad n \in \Nset_0.
\]
To see this, note that if $x \in \cM_n$ and $k \geq n+2$, then $\alpha_k \alpha_0 (x) = \alpha_0 \alpha_{k-1} (x) = \alpha_0 x$. On the other hand, if $x \in \cM_n$ and $k \geq n$, then $\alpha_{k} \alpha_0^{-1} (x) = \alpha_0^{-1} \alpha_{k+1} (x) =\alpha_0^{-1}(x)$. This gives that $\alpha_0^{-1}(\cM_n) \subset \cM_{n-1}$ for $n \geq 1$. Hence, actually $\alpha_0(\cM_n)= \cM_{n+1}$ for all $n \in \Nset_0$. We also note that $\alpha_0^{-1}(\cM_0) \subset \cM_0$.

Thus, generalizing terminology from classical probability, the random variables
\begin{gather*}
\iota_0 := \Id|_{\cM_0} \colon\ \cM_0 \to \cM_0 \subset \cM,
\\
\iota_1 := \alpha_0|_{\cM_0} \colon\ \cM_0 \to \cM_1 \subset \cM,
\\
\iota_2 := \alpha^2_0|_{\cM_0} \colon\ \cM_0 \to \cM_2 \subset \cM,
\\
 \quad\ \ \vdots
 \\
\iota_n := \alpha^n_0|_{\cM_0} \colon\ \cM_0 \to \cM_n \subset \cM
\end{gather*}
are adapted to the filtration $\cM_0 \subset \cM_1 \subset \cM_2 \subset \cdots $, and $\alpha_0$
is the time evolution of the stationary process $(\cM,\psi, \alpha_0, \cM_0)$. An immediate question is whether a representation of the
Thompson group $F$ restricts to the von Neumann subalgebra $\cM_\infty$.

\begin{Definition}\label{definition:generating}
The representation $\rho \colon F \to \Aut(\cM,\psi)$ is said to have the \emph{generating
property} if $\cM_\infty = \cM$.
\end{Definition}

As shown in Proposition~\ref{proposition:generating-property} below, this generating property
entails that each intersected fixed point algebra $\cM_n = \bigcap_{k \ge n+1} \cM^{\alpha_k}$
equals the single fixed point algebra $\cM^{\alpha_{n+1}}$. Thus the generating property
tremendously simplifies the form of the tower $\cM_0 \subset \cM_1 \subset \cdots $, and our next
result shows that this can always be achieved by restriction.

\begin{Proposition}\label{proposition:generatingrestriction}
The representation $\rho\colon F\to \Aut(\cM,\psi)$ restricts to the generating representation
$\rho_{\gen}\colon F\to \Aut(\cM_{\infty},\psi_\infty)$ such that $\alpha_n (\cM_\infty) \subset
\cM_\infty$ and $E_{\cM_\infty} E_{\cM^{\alpha_n}} = E_{\cM^{\alpha_n}} E_{\cM_\infty}$ for all
$n \in \Nset_0$. Here $\psi_\infty$ denotes the restriction of the state $\psi$ to $\cM_\infty$.
$E_{\cM^{\alpha_n}}$ and $E_{\cM_{\infty}}$ denote the unique $\psi$-preserving normal
conditional expectations onto $\cM^{\alpha_n}$ and $\cM_{\infty}$ respectively.
\end{Proposition}

\begin{proof}
We show that $\alpha_i (\cM_n ) \subset \cM_{n+1}$ for all $i,n \ge 0$. Let $x \in \cM_n$. If
$i \ge n+1$ then $\alpha_i (x) = x$ is immediate from the definition of $\cM_n$. If $i < n+1$
then, using the relations for the generators of the Thompson group,
$\alpha_i(x)= \alpha_i\alpha_{k+1} (x) = \alpha_{k+2}\alpha_i(x)$
for any $k \ge n$, thus $\alpha_i(x) \in \cM_{n+1}$. Consequently, $\alpha_i$ maps
$\bigcup_{n \ge 0} \cM_n$ into itself for any $i \in \Nset_0$.
It is also easily verified that $\alpha_i^{-1}(\cM_n) \subset \cM_n$ for all $i$ and $n \ge 0$.
Now a standard approximation
argument shows that $\cM_\infty$ is invariant under $\alpha_i$ and $\alpha_i^{-1}$ for any $i \in \Nset_0$.
Consequently, the representation~$\rho$ restricts to $\cM_\infty$ and, of course, this restriction
$\rho_{\gen}$ has the generating property.

Since $\cM_\infty$ is globally invariant under the modular automorphism group of $(\cM,\psi)$,
there exists the (unique) $\psi$-preserving normal conditional expectation $E_{\cM_\infty}$ from
$\cM$ onto $\cM_\infty$. In~particular, $\rho_{\gen}(g_n) = \alpha_n|_{\cM_\infty}$ commutes
with the modular automorphism group of $(\cM_\infty, \psi_\infty)$ which ensures
$\rho_{\gen}(g_n)\in \Aut(\cM_{\infty},\psi_\infty)$. Finally, that $E_{\cM_\infty}$ and
$E_{\cM^{\alpha_n}}$ commute is concluded from
\[
E_{\cM_\infty} \alpha_n E_{\cM_\infty} = \alpha_n E_{\cM_\infty},
\]
which implies $ E_{\cM^{\alpha_n}} E_{\cM_\infty} = E_{\cM_\infty}E_{\cM^{\alpha_n}}$ by
routine arguments, and an application of the mean ergodic theorem (see for example
\cite[Theorem 8.3]{Ko10}),
\[
 E_{\cM^{\alpha_n}} = \lim_{N \to \infty} \frac{1}{N} \sum_{i=1}^{N} \alpha_n^i,
\]
where the limit is taken in the pointwise strong operator topology.
\end{proof}

\begin{Lemma}\label{lemma:generating}
With the notations as above, $\cM_k=\cM^{\alpha_{k+1}}\cap \cM_{\infty}$ for all $k\in \Nset_0$.
\end{Lemma}

\begin{proof}
For the sake of brevity of notation, let $Q_n=E_{\cM^{\alpha_n}}$ denote the $\psi$-preserving
normal conditional expectation from $\cM$ onto $\cM^{\alpha_n}$. Let us first make the following observation: if $x\in \cM_{\infty}$, then $Q_n(x) \in \cM_{\infty}$ for every $n \in \Nset_0$. Indeed, by Proposition~\ref{proposition:generatingrestriction},
$x\in \cM_\infty$ implies $\alpha_n(x) \in \cM_\infty$ and thus
$\frac{1}{M} \sum_{i=1}^M \alpha_n^i (x)\in \cM_\infty$
for all $M \ge 1$. As $Q_n(x) = \lim_{M \to \infty} \frac{1}{M} \sum_{i=1}^M \alpha_k^i(x)$ in the strong operator topology, this ensures $Q_n(x) \in \cM_\infty$.

By the definition of $\cM_k$
and $\cM_\infty$, it is clear that $\cM_k\subset\cM^{\alpha_{k+1}}\cap \cM_{\infty}$. In order to
show the reverse inclusion, it suffices to show that $Q_n Q_k |_{\cM_\infty}=Q_k |_{\cM_\infty}$ for $
0\leq k<n<\infty$. We claim that, for $0\leq k < n$,
\[
	Q_n Q_k |_{\cM_\infty} = Q_k |_{\cM_\infty} \, \Longleftrightarrow \, Q_k Q_n Q_k |_{\cM_\infty} = Q_k |_{\cM_\infty}.
\]
Indeed this equivalence is immediate from
\begin{align*}
	\psi\big( (Q_nQ_k - Q_k)(y^*) (Q_nQ_k - Q_k)(x)\big)
	& = \psi\big( y^* (Q_kQ_n - Q_k) (Q_nQ_k - Q_k)(x)\big) \\
	& = \psi\big( y^* (Q_k - Q_kQ_n Q_k)(x)\big)
\end{align*}
for all $x,y \in \cM_\infty$. We are left to prove $Q_k Q_n Q_k |_{\cM_\infty}= Q_k
|_{\cM_\infty}$ for $k < n$. For this purpose we express the conditional
expectations $Q_k$ and $Q_n$ as mean ergodic limits in the pointwise strong operator topology
and calculate
\begin{align*}
Q_{k}Q_{n}Q_{k}|_{\cM_\infty}
&=\lim_{M\to \infty}\lim_{N\to \infty}
 \frac{1}{MN}\sum_{i=1}^{M}\sum_{j=1}^{N} \alpha_{k}^{i}\alpha_{n}^{j}Q_{k}|_{\cM_\infty}\\
&=\lim_{M\to \infty}\lim_{N\to\infty}
 \frac{1}{MN} \sum_{i=1}^{M}\sum_{j=1}^{N}\alpha_{n+i}^{j}\alpha_{k}^{i}Q_{k}|_{\cM_\infty}\\
&=\lim_{M\to \infty}\lim_{N\to\infty}\frac{1}{MN}
 \sum_{i=1}^{M}\sum_{j=1}^{N}\alpha_{n+i}^{j}Q_{k}|_{\cM_\infty}\\
&=\lim_{M\to \infty}\frac{1}{M}
 \sum_{i=1}^{M}Q_{n+i}Q_{k}|_{\cM_\infty} = Q_{k}|_{\cM_{\infty}}.
\end{align*}
The last equality is ensured as $x \in \cM_{\infty}$ implies that $Q_k(x) \in \cM_{\infty}$, hence as
$\cM^{\rho(F)} \subset \cM_0 \subset \cdots \subset \cM_{\infty} =\vee_{n \geq 0} \cM_n$, there exists sufficiently large $i_0$ such that $Q_{n+i} Q_k(x)=Q_k (x)$ for all $i \geq i_0$. Thus
{\samepage
\[
\lim_{M\to\infty}\frac{1}{M}\sum_{i=1}^{M}Q_{n+i}Q_k|_{\cM_{\infty}}=\operatorname{Id}Q_k|_{\cM_{\infty}}
\]
in the pointwise strong operator topology.}
\end{proof}

\begin{Corollary}\label{corollary:commsquare}
With notations as introduced at the beginning of the present Section~$\ref{section:Mark-From-Rep}$, the following set of inclusions forms a commuting square for every $n\in \Nset_0$:
\[
	\begin{matrix}
	\cM^{\alpha_{n+1}} &\subset &\cM\\
	\cup & & \cup \\
	\cM_{n} & \subset & \cM_{\infty}.
	\end{matrix}
\]
\end{Corollary}

{\sloppy\begin{proof}
Let $Q_n$ and $E_{\cM_{\infty}}$ be the $\psi$-preserving normal conditional expectations from
$\cM$ onto~$\cM^{\alpha_n}$ and $\cM_{\infty}$ respectively for $n\in \Nset_0$. For
$n\in \Nset_0$, by Proposition~\ref{proposition:generatingrestriction},
$Q_{n+1}E_{\cM_{\infty}}=E_{\cM_{\infty}}Q_{n+1}$ and by Lemma \ref{lemma:generating},
$\cM_n=\cM^{\alpha_{n+1}}\cap \cM_{\infty}$. By $(iv)$ of Proposition~\ref{proposition:cs},
we get a~commuting square.
\end{proof}}

\begin{Proposition}\label{proposition:generating-property}
If the representation $\rho\colon F\to \Aut(\cM,\psi)$ has the generating property then the
following equality holds for all $n \in \Nset_0$:
\[
\cM_n = \cM^{\rho(g_{n+1})}.
\]
In other words, one has the tower of fixed point algebras
\[
\cM^{\rho(F^+)} \subset \cM^{\rho(g_{0})} \subset \cM^{\rho(g_{1})} \subset \cM^{\rho(g_{2})}
\subset \cdots \subset \cM = \bigvee_{n \ge 0} \cM^{\rho(g_n)}.
\]
\end{Proposition}

\begin{proof}
If the representation $\rho$ is generating, then $\cM_\infty=\cM$. Hence
$\cM_n=\cM^{\alpha_{n+1}}$ for all $n\in \Nset_0$ as a consequence of Lemma
\ref{lemma:generating}.
\end{proof}

The following intertwining property will be crucial for obtaining stationary Markov processes from
representations of the Thompson group $F$.

\begin{Proposition}
Suppose $\rho \colon F \to \Aut(\cM,\psi)$ is a $($not necessarily generating$)$ representation
of $F$. Then with $\alpha_n=\rho(g_n)$, the following equality holds:
\[
\alpha_k Q_n = Q_{n+1} \alpha_k \qquad \text{for all}\quad 0 \le k < n < \infty.
\]
 Here $Q_n$ denotes the $\psi$-preserving normal conditional
expectation from $\cM$ onto the fixed point algebra $\cM^{\alpha_n}$ of the represented
generator $\alpha_n \in \Aut(\cM,\psi)$.
\end{Proposition}

\begin{proof}
An application of the mean ergodic theorem and the relations between the generators of the
Thompson group $F$ yield that, for $k < n$,
\[
\alpha_k Q_n
=\lim_{N \to \infty} \frac{1}{N} \sum_{i=1}^N \alpha_k \alpha_n^i=
\lim_{N \to \infty} \frac{1}{N} \sum_{i=1}^N \alpha_{n+1}^i \alpha_k
= Q_{n+1} \alpha_k.
\]
Here the limits are taken in the pointwise strong operator topology.
\end{proof}

\subsection{Commuting squares and Markovianity for stationary processes}\label{subsection:Markov-F}

Given the representation $\rho \colon F \to \Aut(\cM,\psi)$, with represented generators
$\alpha_n := \rho(g_n)$, for $n \in \Nset_0$,
we recall that
\[
\cM_n = \bigcap_{k \ge n+1} \cM^{\alpha_k},
\]
denotes the intersected fixed point algebras.
Throughout this section, let $\cA_0$ be a $\psi$-conditioned von Neumann subalgebra of $\cM_0$. Then
$(\cM,\psi, \alpha_0, \cA_0)$ is a (bilateral noncommutative) stationary process with
generating algebra $\cA_0$ (as introduced in Definition \ref{definition:process-sequence}). Its canonical local filtration is denoted by
$\cA_\bullet \equiv \{\cA_{I}\}_{I\in \cI(\Zset)}$, where
\[
\cA_I := \bigvee_{i \in I}\alpha_0^i(\cA_0),
\]
and an ``interval'' $I \in \cI(\Zset)$ is written as
$[m,n] := \{i \in \Zset \mid m \le i \le n\}$ or
$[m,\infty) := \{i \in \Zset \mid m \le i \}$ or
$(-\infty, n] := \{i \in \Zset \mid i \le n \}$.
Furthermore, $P_I$ will denote the $\psi$-preserving normal conditional expectation
from $\cM$ onto $\cA_I$. Note that the endomorphism $\alpha_0$ acts compatibly on the local
filtration, i.e.,~$\alpha_0(\cA_I) = \cA_{I+1}$ for all $I \in \cI(\Zset)$, where
$I+1 := \{i+1\mid i \in I\}$.

We record a simple, but important, observation obtained from the relations of $F$ on stationary
processes to which we will frequently appeal.

\begin{Proposition}\label{proposition:fixed-point}
Let $(\cM,\psi,\alpha_0,\cA_0)$ be the $($bilateral noncommutative$)$ stationary process with $\cA_0$ a $\psi$-conditioned subalgebra of $\cM_0$. Then it holds that $\cA_{(-\infty,n]}\subset \cM_n$ for all
$n \in \Nset_0$.
\end{Proposition}

\begin{proof}
As $\cA_0\subset \cM_0$, it holds that $\alpha_{n}(x) = x$ for any $x \in \cA_0$ and
$n \in \Nset$. Thus using the defining relations of $F$ we get for $0\le k \le n < \ell$,
	\[
	\alpha_{\ell} \alpha_0^k(x)
	= \alpha_0^k \alpha_{\ell-k}(x) = \alpha_0^k(x).
	\]
On the other hand, for $k < 0$ and $\ell \ge 1$,
	\[
	\alpha_{\ell} \alpha_0^k (x) = \alpha_0^k \alpha_{\ell-k} (x) = \alpha_0^k (x).
	\]
Hence
\begin{gather*}
\cA_{(-\infty,n]} = \bigvee_{i \in (-\infty, n]}\alpha_0^i(\cA_0) \subset \cM_0 \subset \cM_{n} \qquad \text{for all} \quad n\in \Nset_0. \tag*{\qed}
\end{gather*}
\renewcommand{\qed}{}
\end{proof}

We next observe that the generating property of the representation $\rho$ can be concluded
from the minimality of a stationary process.

\begin{Proposition}\label{proposition:minimality-generating}
Suppose the representation $\rho\colon F\to \Aut(\cM,\psi)$ and $\cA_0\subset \cM_0$ are given. If~the stationary process $(\cM,\psi,\alpha_0,\cA_0)$ is minimal, then $\rho$ is generating.
\end{Proposition}

\begin{proof}
For the stationary process $(\cM,\psi,\alpha_0,\cA_0)$, recall that $\cA_{(-\infty,\infty)}
=\bigvee_{i \in \Zset}\alpha_0^i(\cA_0)$ and minimality implies $\cA_{(-\infty,\infty)}=\cM$.
By Proposition~\ref{proposition:fixed-point}, $\cA_{(-\infty,n]}\subset \cM_n$ for all $n \in \Nset_0$.
Thus $\cM = \bigvee_{n \ge 0}\cA_{(-\infty,n]}\subset \bigvee_{n \ge 0} \cM_n = \cM_\infty$. We
conclude from this that the representation $\rho$ has the generating property,
i.e.,~$\cM_\infty = \cM$.
\end{proof}

In the following results, it is not assumed that the stationary process is minimal or that the
representation $\rho$ is generating unless explicitly mentioned.

\begin{Theorem}\label{theorem:markov-filtration-1}
Suppose $\rho \colon F \to \Aut(\cM,\psi)$ is a representation with $\alpha_n:=\rho(g_n)$ as before. Let $\cA_0 \subset \cM_{0}$ and $\cA_{[0,\infty)}
:= \bigvee_{n \in \Nset_0}\alpha_0^n(\cA_0) $ be von Neumann subalgebras of $(\cM,\psi)$ such
that the inclusions
\[
\begin{matrix}
\cM^{\alpha_1} &\subset &\cM \\
\cup & & \cup \\
\cA_0 & \subset & \cA_{[0,\infty)}
\end{matrix}
\]
form a commuting square. Then the family of von Neumann subalgebras
$\cA_\bullet \equiv \{\cA_{I}^{}\}_{I\in \cI(\Zset)}^{}$, with
\[
\cA_I := \bigvee_{i \in I}\alpha_0^i(\cA_0),
\]
is a local Markov filtration and $(\cM,\psi, \alpha_0, \cA_0)$ is a
$($bilateral$)$ stationary Markov process.
\end{Theorem}

\begin{proof}
{\sloppy
Let $Q_n$ and $P_I$ denote the $\psi$-preserving normal conditional expectations from $\cM$ onto~$\calM^{\alpha_n}$ and $\cA_I$ respectively. Note that the commuting square condition
implies $Q_1 P_{[0,\infty)} = P_{[0,0]}$. From Proposition~\ref{proposition:fixed-point},
$\cA_{(-\infty,0]} \subset \cM_0 \subset \cM^{\alpha_{1}}$. Hence we get
\begin{alignat*}{3}
P_{(-\infty,0]} P_{[0,\infty)}
&= P_{(-\infty,0]} Q_{1} P_{[0,\infty)}
 &\qquad & \text{(since $\cA_{(-\infty,0]} \subset \cM^{\alpha_{1}}$)}\\
&= P_{(-\infty,0]} P_{[0,0]} P_{[0,\infty)}
 & \qquad & \text{(by commuting square condition)}\\
&= P_{[0,0]}
 & \qquad & \text{(as $\cA_{[0,0]} \subset \cA_{(-\infty,0]}$ and $\cA_{[0,0]} \subset \cA_{[0,\infty)}$)}.
\end{alignat*}}
Thus, by Lemma \ref{lemma:Mark-Suff}, $\{\cA_{I}\}_{I \in \cI(\Zset)}$ is a local Markov filtration and $(\cM, \psi, \alpha_0, \cA_0)$ is a bilateral stationary Markov process.
\end{proof}

\begin{Corollary}\label{corollary-markov-filtration-0}
Suppose $\rho\colon F \to \Aut(\cM,\psi)$ is a representation with $\alpha_0 = \rho(g_0)$.
Then the quadruple $(\cM,\psi, \alpha_0, \cM_{0})$ is a bilateral stationary Markov process.
\end{Corollary}

\begin{proof}
We know from Corollary \ref{corollary:commsquare} that the following is a commuting square:
\[
\begin{matrix}
\cM^{\alpha_1} &\subset &\cM \\
\cup & & \cup \\
\cM_0 & \subset & \cM_{\infty}.
\end{matrix}
\]
Let $\{\cM_I\}_{I\in \cI(\Zset)}$ denote the local filtration given by $\cM_I=\bigvee_{i\in I}
\alpha_0^{i}(\cM_0)$ and $P_I$ be the corresponding conditional expectations. As
$\cM_{(-\infty,n]}\subset \cM_{n}$ for all $n\in \Nset_0$, it is easily verified that
$\cM_{(-\infty,\infty)}\subset \cM_\infty$. Let $P_0:=P_{[0,0]}$ be the $\psi$-preserving conditional
expectation from $\cM$ onto~$\cM_0$. Then from the commuting square above, we have
$E_{\cM_{\infty}}Q_1=P_{0}$, where $E_{\cM_\infty}$ is of course the conditional expectation
onto $\cM_\infty$. This in turn gives $P_{(-\infty,\infty)}Q_1=P_{(-\infty,\infty)} E_{\cM_{\infty}}Q_1 =
P_{(-\infty,\infty)} P_0=P_0$. Hence we get that $\cM_0$ is a von Neumann subalgebra of~$\cM$ such that
\[
	\begin{matrix}
	\cM^{\alpha_1} &\subset &\cM \\
	\cup & & \cup \\
	\cM_0 & \subset & \cM_{[0,\infty)}
	\end{matrix}
\]
forms a commuting square. By Theorem \ref{theorem:markov-filtration-1},
$(\cM,\psi, \alpha_0,\cM_{0})$ is a stationary Markov process.
\end{proof}

\begin{Corollary}\label{corollary:markov-filtration-0-MN}
Suppose $\rho\colon F \to \Aut(\cM,\psi)$ is a representation with $\alpha_m = \rho(g_m)$,
for $m\in \Nset_0$. Then the quadruple $(\cM,\psi, \alpha_m, \cM_{n})$ is a bilateral stationary Markov process for any $0 \le m \le n < \infty$.
\end{Corollary}

\begin{proof}
Consider the representation $\rho_{m,n} := \rho \circ \sh_{m,n} \colon F \to \Aut(\cM,\psi)$,
where $\sh_{m,n}$ denotes the $(m,n)$-partial shift as introduced in Definition
\ref{definition:mn-shift}. We observe that $\rho_{m,n}(g_0) = \rho(g_m)$ and
$\rho_{m,n}(g_k) = \rho(g_{n+k})$ for all $k \ge 1$. In particular, we get
\[
\bigcap_{k \ge 1}\cM^{\rho_{m,n}(g_k)}
=\bigcap_{k \ge 1} \cM^{\rho(g_{k+n})}
=\bigcap_{k \ge n+1} \cM^{\rho(g_k)}=\cM_n.
\]
Thus Corollary \ref{corollary-markov-filtration-0} applies for the $(m,n)$-shifted representation
$\rho_{m,n}$, and its application completes the proof.
\end{proof}

\begin{Corollary}\label{corollary:markov-filtration-MN}
Suppose $\rho \colon F \to \Aut(\cM,\psi)$ is a generating representation. Then the quadruple
$\big(\cM,\psi, \alpha_m, \cM^{\alpha_{n+1}}\big)$ is a bilateral stationary Markov process for any
$0 \le m \le n < \infty$.
\end{Corollary}

\begin{proof}
If the representation $\rho$ is generating, then $\cM^{\alpha_{n+1}}=\cM_n$. Hence the result follows by Corollary \ref{corollary:markov-filtration-0-MN}.
\end{proof}

\begin{Theorem}\label{theorem:markov-filtration-2}
Let the probability space $(\cM,\psi)$ be equipped with the representation
$\rho \colon F \to \Aut(\cM,\psi)$ and the local filtration
$\cA_\bullet \equiv \{\cA_I\}_{I \in \cI(\Zset)}$, where
$ \cA_I := \bigvee_{i \in I}\rho(g_0^i)(\cA_0)$ for some
$\psi$-conditioned von Neumann subalgebra $\cA_0$ of
$\cM_0 = \bigcap_{k \ge 1} \cM^{\rho(g_k)}
$. Further suppose the inclusions
\[
\begin{matrix}
\cM^{\rho(g_{k+1})} &\subset &\cM \\
\cup & & \cup \\
\cA_{[0,k]} & \subset & \cA_{[0,\infty)}
\end{matrix}
\]
form a commuting square for every $k \ge 0$. Then each cell in the following infinite triangular array of inclusions is a commuting square:
$$
\setcounter{MaxMatrixCols}{20}
\begin{matrix}
\cdots &\hspace{-1.5mm}\subset\hspace{-1.5mm}& \cA_{(-\infty,-2]} & \hspace{-1.5mm}\subset\hspace{-1.5mm} & \cA_{(-\infty,-1]} & \hspace{-1.5mm}\subset\hspace{-1.5mm} & \cA_{(-\infty,0]} & \hspace{-1.5mm}\subset\hspace{-1.5mm}& \cA_{(-\infty,1]} & \hspace{-1.5mm}\subset\hspace{-1.5mm} & \cA_{(-\infty,2]} & \hspace{-1.5mm}\subset\hspace{-1.5mm} & \cdots & \hspace{-1.5mm}\subset\hspace{-1.5mm} & \cA_{(-\infty,\infty)}
\\
&& \cup & & \cup & & \cup & & \cup & & \cup & & \cdots & & \cup
\\
&& \vdots & & \vdots & & \vdots & & \vdots & & \vdots & & \cdots & & \vdots
\\
&& \cup & & \cup & & \cup & & \cup & & \cup & & \cdots & & \cup
\\
&& \cA_{[-2,-2]} & \hspace{-1.5mm}\subset\hspace{-1.5mm} & \cA_{[-2,-1]} & \hspace{-1.5mm}\subset\hspace{-1.5mm} & \cA_{[-2,0]} & \hspace{-1.5mm}\subset\hspace{-1.5mm}& \cA_{[-2,1]} & \hspace{-1.5mm}\subset\hspace{-1.5mm} & \cA_{[-2,2]} & \hspace{-1.5mm}\subset\hspace{-1.5mm} & \cdots & \hspace{-1.5mm}\subset\hspace{-1.5mm} & \cA_{[-2,\infty)}
\\
&&&& \cup && \cup & & \cup & & \cup & & & & \cup
\\
&&&& \cA_{[-1,-1]} & \hspace{-1.5mm}\subset\hspace{-1.5mm} & \cA_{[-1,0]} & \hspace{-1.5mm}\subset\hspace{-1.5mm}& \cA_{[-1,1]} & \hspace{-1.5mm}\subset\hspace{-1.5mm} & \cA_{[-1,2]} & \hspace{-1.5mm}\subset\hspace{-1.5mm} & \cdots & \hspace{-1.5mm}\subset\hspace{-1.5mm} & \cA_{[-1,\infty)}
 \\
&&&&&& \cup & & \cup & & \cup & & & & \cup
\\
&&&&&& \cA_{[0,0]} &\hspace{-1.5mm}\subset\hspace{-1.5mm}& \cA_{[0,1]} & \hspace{-1.5mm}\subset\hspace{-1.5mm} & \cA_{[0,2]} & \hspace{-1.5mm}\subset\hspace{-1.5mm} & \cdots & \hspace{-1.5mm}\subset\hspace{-1.5mm} & \cA_{[0,\infty)}
\\
&& &&&& && \cup & & \cup & & & & \cup
 \\
&& &&&& && \cA_{[1,1]}&\hspace{-1.5mm}\subset\hspace{-1.5mm}&\cA_{[1,2]}&\hspace{-1.5mm}\subset\hspace{-1.5mm}& \cdots & \hspace{-1.5mm}\subset\hspace{-1.5mm} & \cA_{[1,\infty)}
\\
 && && &&&& & & \cup & & & & \cup
 \\
 && &&&& &&&& \cA_{[2,2]} & \hspace{-1.5mm}\subset\hspace{-1.5mm} & \cdots & \hspace{-1.5mm}\subset\hspace{-1.5mm} & \cA_{[2,\infty)}
 \\
&&&&&&&& &&&& && \cup
 \\
&&&&&&&& &&&& && \vdots
\end{matrix}
\setcounter{MaxMatrixCols}{10}
$$
In particular, $\cA_\bullet$ is a local Markov filtration.
\end{Theorem}

\begin{proof}
All claimed inclusions in the triangular array are clear from the definition of $\cA_{[m,n]}$.
We recall from Proposition~\ref{proposition:fixed-point} that
$\alpha_0^{k}(\cA_0)\subset \cM^{\alpha_{n+1}}$ for $k \leq n$.
Hence $\cA_{[m,n]} \subset \cM^{\alpha_{n+1}}$ for all $m \le n$. Next we show that, for
 $-\infty< m < n < \infty$, the cell of inclusions
\[
\begin{matrix}
\cA_{[m,n]} &\subset &\cA_{[m,n+1]} \\
\cup & & \cup \\
\cA_{[m+1,n]} & \subset & \cA_{[m+1,n+1]}
\end{matrix}
\]
forms a commuting square. So, as $P_I$ denotes the normal $\psi$-preserving conditional
expectation from $\cM$ onto $\cA_I$, we need to show
$P_{[m,n]} P_{[m+1,n+1]} = P_{[m+1,n]}$.
As $\alpha_0^m P_I \alpha_0^{-m} = P_{I+m}$ for all ${m \in \Zset}$, it suffices to show that, for all ${n \in \Nset}$,
$P_{[0,n]} P_{[1,n+1]} = P_{[1,n]}$
or, equivalently,
$P_{[0,n]} \alpha_0 P_{[0,n]} = \alpha_0 P_{[0,n-1]}$.
We calculate
\begin{align*}
P_{[0,n]} \alpha_0 P_{[0,n]}
 &= P_{[0,n]} Q_{n+1}\alpha_0 P_{[0,n]}
= P_{[0,n]} \alpha_0 Q_{n} P_{[0,n]}\\
& = P_{[0,n]} \alpha_0 Q_{n} P_{[0,\infty)}P_{[0,n]}
= P_{[0,n]} \alpha_0 P_{[0,n-1]} P_{[0,n]} \\
& = P_{[0,n]} \alpha_0 P_{[0,n-1]}=\alpha_0 P_{[0,n-1]}.
\end{align*}

Here we have used $ P_{[0,n]} = P_{[0,n]} Q_{n+1}$, the intertwining properties of
$\alpha_0$ and the commuting square assumption $ Q_{n} P_{[0,\infty)}= P_{[0,n-1]}$.
Thus each cell of inclusions in this triangular array forms a commuting square.
\end{proof}

More generally, we may consider a probability space which is equipped both with a local filtration and
a representation of the Thompson group $F$, and formulate compatiblity conditions between the local
filtration and the representation such that one obtains rich commuting square structures.

\begin{Corollary} \label{corollary:triangulararray}
Suppose the probability space $(\cM,\psi)$ is equipped with a local filtration
$\cN_\bullet \equiv \{\cN_{I}^{}\}_{I \in \cI(\Zset)}^{}$
and a representation $\rho \colon F \to \Aut(\cM,\psi)$ such that
\begin{enumerate}\itemsep=0pt
\item[$(i)$] 
$\rho(g_0) (\cN_I)= \cN_{I+1}$ for all $I \in \cI(\Zset)$ $($compatibility$)$,
\item[$(ii)$] 
$\cN_{[0,n]} \subset \cM^{\rho(g_{n+1})}$ for all $n \in \Nset_0$ $($adaptedness$)$,
\item[$(iii)$] 
the inclusions
\[
\begin{matrix}
\cM^{\rho(g_{k+1})} &\subset &\cM \\
\cup & & \cup \\
\cN_{[0,k]} & \subset & \cN_{[0,\infty)}
\end{matrix}
\]
form a commuting square for all $k \in \Nset_0$.
\end{enumerate}
Then each cell in the following infinite triangular array of inclusions is a commuting square:
$$
\setcounter{MaxMatrixCols}{20}
\begin{matrix}
\cdots& \hspace{-2mm}\subset\hspace{-2mm}& \cN_{(-\infty,-2]} & \hspace{-2mm}\subset\hspace{-2mm} & \cN_{(-\infty,-1]} & \hspace{-2mm}\subset\hspace{-2mm} & \cN_{(-\infty,0]} & \subset& \cN_{(-\infty,1]} & \hspace{-2mm}\subset\hspace{-2mm} & \cN_{(-\infty,2]} & \hspace{-2mm}\subset\hspace{-2mm} & \cdots & \hspace{-2mm}\subset\hspace{-2mm} & \cN_{(-\infty,\infty)}
\\
&& \cup & & \cup & & \cup & & \cup & & \cup & & \cdots & & \cup
\\
&& \vdots & & \vdots & & \vdots & & \vdots & & \vdots & & \cdots & & \vdots
\\
&& \cup & & \cup & & \cup & & \cup & & \cup & & \cdots & & \cup
\\
&& \cN_{[-2,-2]} & \hspace{-2mm}\subset\hspace{-2mm} & \cN_{[-2,-1]} & \hspace{-2mm}\subset\hspace{-2mm} & \cN_{[-2,0]} & \subset& \cN_{[-2,1]} & \hspace{-2mm}\subset\hspace{-2mm} & \cN_{[-2,2]} & \hspace{-2mm}\subset\hspace{-2mm} & \cdots & \hspace{-2mm}\subset\hspace{-2mm} & \cN_{[-2,\infty)}
\\
&&&& \cup && \cup & & \cup & & \cup & & & & \cup
\\
&&&& \cN_{[-1,-1]} & \hspace{-2mm}\subset\hspace{-2mm} & \cN_{[-1,0]} & \subset& \cN_{[-1,1]} & \hspace{-2mm}\subset\hspace{-2mm} & \cN_{[-1,2]} & \hspace{-2mm}\subset\hspace{-2mm} & \cdots & \hspace{-2mm}\subset\hspace{-2mm} & \cN_{[-1,\infty)}
 \\
&&&&&& \cup & & \cup & & \cup & & & & \cup
\\
&&&&&& \cN_{[0,0]} &\subset& \cN_{[0,1]} & \hspace{-2mm}\subset\hspace{-2mm} & \cN_{[0,2]} & \hspace{-2mm}\subset\hspace{-2mm} & \cdots & \hspace{-2mm}\subset\hspace{-2mm} & \cN_{[0,\infty)}
\\
&& &&&& && \cup & & \cup & & & & \cup
 \\
&& &&&& && \cN_{[1,1]}&\subset&\cN_{[1,2]}&\subset& \cdots & \hspace{-2mm}\subset\hspace{-2mm} & \cN_{[1,\infty)}
 \\
& & && &&&& & & \cup & & & & \cup
 \\
& & &&&& &&&& \cN_{[2,2]}
 & \hspace{-2mm}\subset\hspace{-2mm} & \cdots & \hspace{-2mm}\subset\hspace{-2mm} & \cN_{[2,\infty)}
 \\
&&&&&&&& &&&& && \cup
 \\
&&&&&&&& &&&& && \vdots
\end{matrix}
\setcounter{MaxMatrixCols}{10}
$$
In particular, $\cN_\bullet$ is a local Markov filtration.
\end{Corollary}

\begin{proof}Let $P_I$ be the normal $\psi$-preserving conditional expectation onto $\cN_I$. Let
$\alpha_n=\rho(g_n)$ and $Q_n$ be the normal $\psi$-preserving conditional expectation onto
$\cM^{\alpha_{n}}$ as before. We observe that $\cN=\cN_{[0,0]}\subset \cM^{\alpha_1}$ by the
adaptedness condition $(ii)$. This adaptedness property also gives us $\cN_{[0,n]}\subset
\cM^{\alpha_{n+1}}$, and thus $P_{[0,n]}=P_{[0,n]}Q_{n+1}$, for any $n \in \Nset_0$.
The rest of the proof follows the
arguments used in the proof of Theorem~\ref{theorem:markov-filtration-2}.
\end{proof}

\section[Constructions of representations of F from stationary Markov processes]{Constructions of representations of $\boldsymbol{F}$\\ from stationary Markov processes}\label{section:Reps-of-F-from-Mark}

This section is about how to construct representations of the Thompson group $F$ as they arise in noncommutative probability theory. It will be seen that a large class of bilateral stationary Markov processes in tensor dilation form (see Definition~\ref{definition:tensordilation}) will give rise to representations of~$F$. In particular, this will establish that a Markov map on a probability space $(\cA, \varphi)$ with~$\cA$ a~commutative von Neumann algebra can be written as a compressed represented generator~of~$F$.

\subsection{An illustrative example} \label{subsection:Example}
Let $(\cA,\varphi)$ and $(\cC,\chi)$ be noncommutative probability spaces.
We have already shown in \cite{KKW20} how
to obtain a representation of the Thompson monoid $F^+$ and a unilateral stationary Markov process on
$\big(\cA \otimes \cC^{\otimes_{\Nset_0}} , \varphi \otimes \chi^{\otimes_{\Nset_0}}\big)$.
In general, especially for $\cC$
finite-dimensional, this
tensor product model for a noncommutative probability space is ``too small'' to accommodate a representation of the Thompson group $F$. Also, even though the extension
$\big(\cA \otimes \cC^{\otimes_{\Zset}} , \varphi \otimes \chi^{\otimes_{\Zset}}\big)$ suffices to set up a bilateral extension
of a unilateral stationary Markov process (see for example \cite[Section~4.2.2]{Ku85}), it would still be ``too small'' for canonically extending a represention of the monoid $F^+$ to one of the group $F$.

This motivates the following model build on
two given noncommutative probability spaces $(\cA,\varphi)$ and $(\cC,\chi)$. Throughout this final section, consider the infinite von Neumann
algebraic tensor product with respect to an infinite tensor product state given by
\[
(\cM, \psi)
:= \big(\cA \otimes \cC^{\otimes_{\Nset_0^2}} , \varphi \otimes \chi^{\otimes_{\Nset_0^2}}\big).
\]
This probability space can be equipped with a representation of the Thompson group $F$. Also it can
be used to set up a bilateral noncommutative Bernoulli shift and, more generally, a bilateral stationary noncommutative Markov process. We start with providing
a representation of the Thompson group $F$.

For $k \in \Nset_0$, let $\beta_k$ be the automorphisms of $\cM$ defined on the weak*-total set of finite elementary tensors in $\cM$ as
\[
\beta_0\biggl(a \otimes \biggl(\bigotimes_{(i,j)\in \Nset_0^2} x_{i,j}\biggr)\biggr):= a \otimes \biggl(\bigotimes_{(i,j)\in \Nset_0^2} y_{i,j}\biggr)
\qquad
\text{with}
\quad
y_{i,j} = \begin{cases}
x_{2i+1,j} & \text{if}\ j=0, \\
x_{2i,j-1} & \text{if}\ j=1, \\
x_{i,j-1} & \text{if}\ j \geq 2
\end{cases}
\]
and
\[
\beta_k\biggl(a \otimes \biggl(\bigotimes_{(i,j)\in \Nset_0^2} x_{i,j}\biggr)\biggr):= a \otimes \biggl(\bigotimes_{(i,j)\in \Nset_0^2} y_{i,j}\biggr)
\qquad
\text{with}
\quad
y_{i,j} = \begin{cases}
x_{i,j} & \text{if}\ j\le k-1, \\
x_{2i+1,j} & \text{if}\ j=k, \\
x_{2i,j-1} & \text{if}\ j=k+1, \\
x_{i,j-1} & \text{if}\ j \geq k+1
\end{cases}
\]
for $k \in \Nset$. It is evident from these two definitions that
the actions of $\beta_0$ and $\beta_1$ are induced from
corresponding shifts on the index set $\Nset_0^2$, as visualized graphically in Figure~\ref{figure:graphs}.

\begin{figure}[ht]\centering
{\color{blue}$\beta_0 \quad \hat{=} $ }
\begin{tikzcd}[column sep=0.8em, row sep=0.5em]
& \vdots & \vdots& \vdots & \vdots \phantom{\,\,\cdots} \\[1pt]
& \arrow[ddddr, blue, start anchor={[xshift=-1.5ex,yshift=2.8ex]}, end anchor={[xshift=0.5ex,yshift=-0.8ex]}] \arrow[dddd, blue, start anchor={[xshift=-0.2ex,yshift=2.8ex]}, end anchor={[xshift=0.8ex,yshift=-0.8ex]}, bend right] \arrow[dddr, blue, start anchor={[xshift=1ex,yshift=3ex]}, end anchor={[xshift=0.5ex,yshift=-0.8ex]}] \arrow[ddd, blue, start anchor={[xshift=-0.8ex,yshift=3ex]}, end anchor={[xshift=0.8ex,yshift=-0.8ex]}, bend right] \arrow[ddr, blue, start anchor={[xshift=3ex,yshift=3ex]}, end anchor={[xshift=0.8ex,yshift=-0.8ex]}]
\arrow[dd, blue, start anchor={[xshift=-1.2ex,yshift=3.2ex]}, end anchor={[xshift=0.8ex,yshift=-0.8ex]}, bend right]
\arrow[d, blue, start anchor={[xshift=-1ex,yshift=3.2ex]}, end anchor={[xshift=0.8ex,yshift=-0.8ex]}, bend right] \arrow[dr, blue, start anchor={[xshift=4ex,yshift=2.6ex]}, end anchor={[xshift=1.5ex,yshift=-0.8ex]}] & & & \\
& \bullet \arrow[dddd, blue, start anchor={[xshift=0.8ex,yshift=1ex]}, end anchor={[xshift=0.8ex, yshift=-0.7ex]}, bend right]
& \bullet \arrow[r, blue, start anchor={[xshift=-1ex]}, end anchor={[xshift=1ex]}]
& \bullet \arrow[r, blue, start anchor={[xshift=-1ex]}, end anchor={[xshift=1ex]}]
& \bullet \,\, \cdots \\
& \bullet \arrow[dddr, blue, start anchor={[xshift=-0.6ex,yshift=0.9ex]}, end anchor={[xshift=0.6ex,yshift=-0.8ex]}]
& \bullet \arrow[r, blue, start anchor={[xshift=-1ex]}, end anchor={[xshift=1ex]}]
& \bullet \arrow[r, blue, start anchor={[xshift=-1ex]}, end anchor={[xshift=1ex]}]
& \bullet \,\, \cdots \\
& \bullet \arrow[ddd, blue, start anchor={[xshift=0.8ex, yshift=1ex]}, end anchor={[xshift=0.7ex,yshift=-0.8ex]}, bend right]
& \bullet \arrow[r, blue, start anchor={[xshift=-1ex]}, end anchor={[xshift=1ex]}]
& \bullet \arrow[r, blue, start anchor={[xshift=-1ex]}, end anchor={[xshift=1ex]}]
& \bullet \,\, \cdots \\
& \bullet \arrow[ddr, blue, start anchor={[xshift=-0.7ex,yshift=1ex]}, end anchor={[xshift=0.8ex,yshift=-0.8ex]}]
& \bullet \arrow[r, blue, start anchor={[xshift=-1ex]}, end anchor={[xshift=1ex]}]
& \bullet \arrow[r, blue, start anchor={[xshift=-1ex]}, end anchor={[xshift=1ex]}]
& \bullet \,\, \cdots \\
& \bullet \arrow[dd, blue, start anchor={[xshift=0.8ex,yshift=1ex]}, end anchor={[xshift=0.9ex,yshift=-0.8ex]}, bend right]
& \bullet \arrow[r, blue, start anchor={[xshift=-1ex]}, end anchor={[xshift=1ex]}]
& \bullet \arrow[r, blue, start anchor={[xshift=-1ex]}, end anchor={[xshift=1ex]}]
& \bullet \,\, \cdots \\
& \bullet \arrow[dr, blue, start anchor={[xshift=-1.3ex,yshift=0.8ex]}, end anchor={[xshift=1.2ex,yshift=-1.0ex]}]
& \bullet \arrow[r, blue, start anchor={[xshift=-1ex]}, end anchor={[xshift=1ex]}]
& \bullet \arrow[r, blue, start anchor={[xshift=-1ex]}, end anchor={[xshift=1ex]}]
& \bullet \,\, \cdots \\
\uparrow{i} & \bullet \arrow[d, blue, start anchor={[yshift=1ex]}, end anchor={[yshift=-0.8ex]}]
& \bullet \arrow[r, blue, start anchor={[xshift=-1ex]}, end anchor={[xshift=1ex]}]
& \bullet \arrow[r, blue, start anchor={[xshift=-1ex]}, end anchor={[xshift=1ex]}]
& \bullet \,\, \cdots\\
\blacksquare
& \bullet \arrow[r, blue, start anchor={[xshift=-1ex]}, end anchor={[xshift=1ex]}]
& \bullet \arrow[r, blue, start anchor={[xshift=-1ex]}, end anchor={[xshift=1ex]}]
& \bullet \arrow[r, blue, start anchor={[xshift=-1ex]}, end anchor={[xshift=1ex]}]
& \bullet \,\, \cdots \\
 & \xrightarrow{j} &&&
\end{tikzcd} \qquad \quad
{\color{blue}$\beta_1 \quad \hat{=} $ }
\begin{tikzcd}[column sep=0.8em, row sep=0.5em]
& \vdots & \vdots & \vdots& \vdots & \vdots \phantom{\,\,\cdots} \\[1pt]
& & \arrow[ddddr, blue, start anchor={[xshift=-1.5ex,yshift=2.8ex]}, end anchor={[xshift=0.5ex,yshift=-0.8ex]}] \arrow[dddd, blue, start anchor={[xshift=-0.2ex,yshift=2.8ex]}, end anchor={[xshift=0.8ex,yshift=-0.8ex]}, bend right] \arrow[dddr, blue, start anchor={[xshift=1ex,yshift=3ex]}, end anchor={[xshift=0.5ex,yshift=-0.8ex]}] \arrow[ddd, blue, start anchor={[xshift=-0.8ex,yshift=3ex]}, end anchor={[xshift=0.8ex,yshift=-0.8ex]}, bend right] \arrow[ddr, blue, start anchor={[xshift=3ex,yshift=3ex]}, end anchor={[xshift=0.8ex,yshift=-0.8ex]}]
\arrow[dd, blue, start anchor={[xshift=-1.2ex,yshift=3.2ex]}, end anchor={[xshift=0.8ex,yshift=-0.8ex]}, bend right]
\arrow[d, blue, start anchor={[xshift=-1ex,yshift=3.2ex]}, end anchor={[xshift=0.8ex,yshift=-0.8ex]}, bend right] \arrow[dr, blue, start anchor={[xshift=4ex,yshift=2.6ex]}, end anchor={[xshift=1.5ex,yshift=-0.8ex]}] & & & \\
&\bullet & \bullet \arrow[dddd, blue, start anchor={[xshift=0.8ex,yshift=1ex]}, end anchor={[xshift=0.8ex, yshift=-0.7ex]}, bend right]
& \bullet \arrow[r, blue, start anchor={[xshift=-1ex]}, end anchor={[xshift=1ex]}]
& \bullet \arrow[r, blue, start anchor={[xshift=-1ex]}, end anchor={[xshift=1ex]}]
& \bullet \,\, \cdots \\
&\bullet & \bullet \arrow[dddr, blue, start anchor={[xshift=-0.6ex,yshift=0.9ex]}, end anchor={[xshift=0.6ex,yshift=-0.8ex]}]
& \bullet \arrow[r, blue, start anchor={[xshift=-1ex]}, end anchor={[xshift=1ex]}]
& \bullet \arrow[r, blue, start anchor={[xshift=-1ex]}, end anchor={[xshift=1ex]}]
& \bullet \,\, \cdots \\
&\bullet & \bullet \arrow[ddd, blue, start anchor={[xshift=0.8ex, yshift=1ex]}, end anchor={[xshift=0.7ex,yshift=-0.8ex]}, bend right]
& \bullet \arrow[r, blue, start anchor={[xshift=-1ex]}, end anchor={[xshift=1ex]}]
& \bullet \arrow[r, blue, start anchor={[xshift=-1ex]}, end anchor={[xshift=1ex]}]
& \bullet \,\, \cdots \\
&\bullet & \bullet \arrow[ddr, blue, start anchor={[xshift=-0.7ex,yshift=1ex]}, end anchor={[xshift=0.8ex,yshift=-0.8ex]}]
& \bullet \arrow[r, blue, start anchor={[xshift=-1ex]}, end anchor={[xshift=1ex]}]
& \bullet \arrow[r, blue, start anchor={[xshift=-1ex]}, end anchor={[xshift=1ex]}]
& \bullet \,\, \cdots \\
&\bullet & \bullet \arrow[dd, blue, start anchor={[xshift=0.8ex,yshift=1ex]}, end anchor={[xshift=0.9ex,yshift=-0.8ex]}, bend right]
& \bullet \arrow[r, blue, start anchor={[xshift=-1ex]}, end anchor={[xshift=1ex]}]
& \bullet \arrow[r, blue, start anchor={[xshift=-1ex]}, end anchor={[xshift=1ex]}]
& \bullet \,\, \cdots \\
&\bullet & \bullet \arrow[dr, blue, start anchor={[xshift=-1.3ex,yshift=0.8ex]}, end anchor={[xshift=1.2ex,yshift=-1.0ex]}]
& \bullet \arrow[r, blue, start anchor={[xshift=-1ex]}, end anchor={[xshift=1ex]}]
& \bullet \arrow[r, blue, start anchor={[xshift=-1ex]}, end anchor={[xshift=1ex]}]
& \bullet \,\, \cdots \\
\uparrow{i} &\bullet & \bullet \arrow[d, blue, start anchor={[yshift=1ex]}, end anchor={[yshift=-0.8ex]}]
& \bullet \arrow[r, blue, start anchor={[xshift=-1ex]}, end anchor={[xshift=1ex]}]
& \bullet \arrow[r, blue, start anchor={[xshift=-1ex]}, end anchor={[xshift=1ex]}]
& \bullet \,\, \cdots\\
\blacksquare
&\bullet & \bullet \arrow[r, blue, start anchor={[xshift=-1ex]}, end anchor={[xshift=1ex]}]
& \bullet \arrow[r, blue, start anchor={[xshift=-1ex]}, end anchor={[xshift=1ex]}]
& \bullet \arrow[r, blue, start anchor={[xshift=-1ex]}, end anchor={[xshift=1ex]}]
& \bullet \,\, \cdots \\
 & \xrightarrow{j} &&&
\end{tikzcd} \qquad \quad
\caption{Visualization of the action of the automorphisms $\beta_0$ (left) and $\beta_1$ (right). Here $\blacksquare$ denotes an element of $\cA$ and $\bullet$ denotes an element of $\cC$, and the blue arrows indicate how the automorphisms act as shifts when considered on the index set $\Nset_0^2$.}
\label{figure:graphs}
\end{figure}
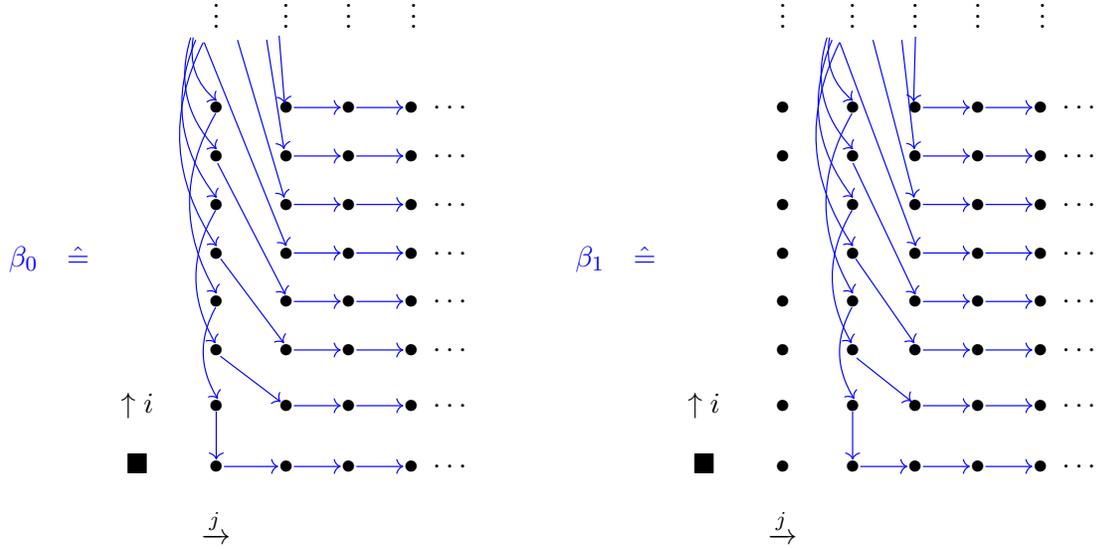
\color{black}

We note that the fixed point algebras
$\cM^{\beta_0}$ and $\cM^{\beta_1}$ of $\beta_0$ and $\beta_1$ are given by, respectively,
\begin{align}
\cM^{\beta_0}
&= \cA \otimes \1_\cC^{\otimes_{\Nset_0}}
 \otimes \1_\cC^{\otimes_{\Nset_0}}
 \otimes \1_\cC^{\otimes_{\Nset_0}}
 \otimes \cdots \label{eq:f-p-a-0}, \\[1ex]
\cM^{\beta_1}
&= \cA \otimes \cC^{\otimes_{\Nset_0}}
 \otimes \1_\cC^{\otimes_{\Nset_0}}
 \otimes \1_\cC^{\otimes_{\Nset_0}}
 \otimes \cdots \label{eq:f-p-a-1}.
\end{align}

Let $\cB_0 := \beta_0^{-1}
\big(\cA \otimes \1_\cC^{\otimes_{\Nset_0}}
\otimes \cC^{\otimes_{\Nset_0}}
\otimes \1_\cC^{\otimes_{\Nset_0}}
\otimes \cdots\big)$ which can be thought of as the ``present'' von Neumann subalgebra at time $n=0$ of the explicit form
\[
\begin{array}{ccccccccc}
&& \vdots & & \vdots & & \vdots & & \\
&& \otimes && \otimes && \otimes && \\
&& \1_\cC & & \1_\cC & & \1_\cC & & \\
&& \otimes & & \otimes & & \otimes & & \\
&& \cC & & \1_\cC & & \1_\cC & & \\
 && \otimes & & \otimes & & \otimes & & \\
 && \1_\cC & & \1_\cC & & \1_\cC & & \\
 &&\otimes & & \otimes & & \otimes & & \\
\cA & \otimes & \cC & \otimes & \1_\cC & \otimes & \1_\cC & \otimes & \cdots.
 \end{array}
 \]
\begin{Proposition}\label{proposition:rep-Fbeta}
The maps $g_n \mapsto \rho_B(g_n) := \beta_n$, with $n \in \Nset_0$, extend multiplicatively
to a~representation $\rho_B \colon F \to \Aut(\cM,\psi)$ which has the generating property. Further,
$(\cM, \psi, \beta_0, \cB_0)$ is a bilateral noncommutative Bernoulli shift with generator
$\cB_0$.
\end{Proposition}
\begin{proof}
For $0 \le k < \ell < \infty$, the relations $\beta_k \beta_\ell = \beta_{\ell+1} \beta_k$
are verified in a straightforward computation on finite elementary tensors. Since $\psi \circ \beta_n = \psi$, the
maps $g_n \mapsto \rho_B(g_n) := \beta_n$ extend to a representation of $F$ in $\Aut(\cM, \psi)$. The generating
property of this representation will follow from the minimality of the stationary process by
Proposition~\ref{proposition:minimality-generating}. Indeed, let $\cB_I := \bigvee_{i \in I}\beta_0^i(\cB_0)$ for
$I\in\cI(\Zset)$ and note that $\cB_{[0,0]} = \cB_0$. Clearly $\cB_{\Zset} = \cM$, hence the stationary
process $(\cM, \psi, \beta_0, \cB_0)$ is minimal. We are left to show that this minimal
stationary process is a bilateral noncommutative Bernoulli shift.
Clearly, $ \cM^{\beta_0} \subset \cB_0$. We are left to verify the factorization
\[Q_0(xy)=Q_0(x)Q_0(y)\] for any $x\in \cB_I, y\in \cB_J$ whenever $I \cap J = \varnothing$.
Here $Q_0$ is the $\psi$-preserving normal conditional expectation from $\cM$ onto
$\cM^{\beta_0}$ which is of the tensor type
\[
Q_0\biggl(a \otimes \biggl(\bigotimes_{(i,j)\in \Nset_0^2} x_{i,j}\biggr)\biggr) = a \otimes \biggl(\bigotimes_{(i,j)\in \Nset_0^2} \chi(x_{i,j})\1_{\cC}\biggr)
\]
for finite elementary tensors in $\cM$. Now the required factorization easily follows by observing that distinct powers of the ``time evolution'' $\beta_0$ send elements of $\cB_0$ to elements which are supported by disjoint index sets in $\Nset_0^2$.
\end{proof}

To obtain more general representations of the Thompson group $F$, we can further ``perturb'' the automorphisms $\beta_n$. Here
we focus on a very particular case of such perturbations, as it will turn out to be useful when constructing representations of $F$ from bilateral stationary noncommutative Markov processes.

Given an automorphism $\gamma \in \Aut(\cA \otimes \cC, \varphi \otimes \chi)$, let $\gamma_0 \in \Aut(\cM, \psi)$ denote its natural extension such that
\[
\gamma_0\biggl(a \otimes \biggl(\bigotimes_{(i,j)\in \Nset_0^2} x_{i,j}\biggr)\biggr) = \gamma(a \otimes x_{00}) \otimes \biggl(\bigotimes_{(i,j)\in \Nset_0^2 \setminus \{(0,0)\}} x_{i,j}\biggr).
\]
Furthermore, let
\[
\alpha_0 := \gamma_0 \circ \beta_0,\qquad
\alpha_n := \beta_{n} \qquad \text{for all}\quad n \ge 1.
\]
\begin{Proposition}\label{proposition:rep-Falpha}
The maps $g_n \mapsto \rho_M(g_n) := \alpha_n $, with $n \in \Nset_0$, extend multiplicatively
to a~representation $\rho_M \colon F \to \Aut(\cM,\psi)$ which has the generating property.
Further, the quadruple $(\cM, \psi, \alpha_0, \cM^{\alpha_1})$ is a bilateral noncommutative stationary Markov process.
\end{Proposition}
\begin{proof}
For $1 \leq k <\ell$, the relations $\alpha_k \alpha_{\ell} = \alpha_{\ell+1} \alpha_k$ are those of the $\beta_n$-s
from Proposition~\ref{proposition:rep-Fbeta}. The relations $\alpha_0 \alpha_{\ell} = \alpha_{\ell+1}\alpha_0$ for
$l >0$ are verified on finite elementary tensors by a straightforward computation. Similar arguments as used in the proof of
Proposition~\ref{proposition:rep-Fbeta} ensure that the maps $g_n \mapsto \rho_M(g_n) := \alpha_n $ extend multiplicatively
to a representation $\rho_M \colon F \to \Aut(\cM,\psi)$. Its generating property is again immediate from the
minimality of the stationary process by Proposition~\ref{proposition:minimality-generating}. Finally, the Markovianity of the
bilateral stationary process $(\cM, \psi, \alpha_0, \cM^{\alpha_1})$ follows from Corollary
\ref{corollary:markov-filtration-MN}.
\end{proof}

Given the stationary Markov process $(\cM, \psi, \alpha_0, \cM^{\alpha_1})$ (from Proposition~\ref{proposition:rep-Falpha}),
a restriction of the generating algebra $\cM^{\alpha_1}$ to a von Neumann subalgebra $\cA_0$ provides a candidate
for another stationary Markov process. Viewing the Markov shift $\alpha_0$ as a ``perturbation'' of the Bernoulli
shift $\beta_0$, the subalgebra $\cA_0 = \cM^{\beta_0}$ is an interesting choice.

\begin{Proposition} \label{proposition:MarkovTwoReps2}
The quadruple $\big(\cM, \psi, \alpha_0, \cM^{\beta_0}\big)$ is a bilateral noncommutative stationary Markov process.
\end{Proposition}

\begin{proof}We recall from \eqref{eq:f-p-a-0} that
\begin{align*}
 \cM^{\beta_0} &= \cA \otimes \1_\cC^{\otimes_{\Nset_0}}
 \otimes \1_\cC^{\otimes_{\Nset_0}}
 \otimes \1_\cC^{\otimes_{\Nset_0}}
 \otimes \cdots.
\end{align*}
{\sloppy
Let $P_I$ denote the $\psi$-preserving normal conditional expectation from $\cM$ onto
$ \cA_{I} := \bigvee_{i \in I} \alpha_0^i\big(\cM^{\beta_0}\big)$ for an interval $I \subset \Zset$.
By Lemma \ref{lemma:Mark-Suff}, it suffices to verify the Markov property
\[
P_{(-\infty,0]}P_{[0,\infty)} = P_{[0,0]}.
\]}
For this purpose we use the von Neumann subalgebra
\[
\cD_0 :=
\begin{array}{ccccccccc}
&& \vdots & & \vdots & & \vdots & & \\
 && \otimes & & \otimes & & \otimes & & \\
 && \1_\cC & & \1_\cC & & \1_\cC & & \cdots \\
 &&\otimes & & \otimes & & \otimes & & \\
\cA & \otimes & \cC & \otimes & \1_\cC & \otimes & \1_\cC & \otimes & \cdots \\
\end{array}
\]
and the tensor shift $\beta_0$ to generate the ``past algebra'' $\cD_{<} := \bigvee_{i < 0} \beta_0^i(\cD_0)$ and
the ``future algebra'' $\cD_{\ge } := \bigvee_{i\ge 0} \beta_0^i(\cD_0)$.
One has the inclusions
\[
\cA_{(-\infty,0]} \subset \cD_{<}, \qquad
\cA_{[0,\infty)} \subset \cD_{\ge}, \qquad
\cD_{<} \cap \cD_{\ge} = \cM^{\beta_0}.
\]
Here we used for
the first inclusion that $\alpha_0 = \gamma_0 \circ \beta_0$ and thus
$\alpha_0^{-1} = \beta_0^{-1} \circ \gamma_0^{-1}$. The second inclusion is
immediate from the definitions of the von Neumann algebras. Finally, the claimed
intersection property is readily deduced from the underlying tensor product structure.
Let $E_{\cD_{<}}$ and $E_{\cD_{\ge}}$ denote the $\psi$-preserving normal conditional expectations from $\cM$ onto
$ \cD_{<} $ and $ \cD_{\ge}$, respectively. We observe that $E_{\cD_{<}} E_{\cD_{\ge}} = P_{[0,0]}$ is immediately
deduced from the tensor product structure of the probability space $(\cM,\psi)$. But this allows us to compute
\begin{gather*}
 P_{(-\infty,0]}P_{[0,\infty)}
 = P_{(-\infty,0]} E_{\cD_{<}} E_{\cD_{\ge}} P_{[0,\infty)}
 = P_{(-\infty,0]} P_{[0,0]} P_{[0,\infty)}
 = P_{[0,0]}.\tag*{\qed}
\end{gather*}
\renewcommand{\qed}{}
\end{proof}

\begin{Remark}
The above constructed bilateral noncommutative stationary Markov process $\big(\cM, \psi, \alpha_0, \cM^{\beta_0}\big)$ is not minimal, as
the von Neumann algebra $\mathcal{A}_{\mathbb{Z}}$ generated by $\alpha_0^n\big(\cM^{\beta_0}\big)$ for all $n \in \Zset$ is clearly contained in the subalgebra
\[
\begin{array}{@{}ccccccccc}
&& \vdots & & \vdots & & \vdots & & \\
&& \otimes && \otimes && \otimes && \\
&& \cC & & \1_\cC & & \1_\cC & & \\
&& \otimes & & \otimes & & \otimes & & \\
&& \cC & & \1_\cC & & \1_\cC & & \\
 && \otimes & & \otimes & & \otimes & & \\
 && \cC & & \1_\cC & & \1_\cC & & \\
 &&\otimes & & \otimes & & \otimes & & \\
\cA & \otimes & \cC & \otimes & \cC & \otimes & \cC & \otimes & \cdots.
 \end{array}
 \]
 The subalgebra $\mathcal{A}_{\mathbb{Z}}$ is invariant under the action of $\alpha_0 = \rho_M(g_0)$ and its inverse, but it fails to be invariant under the action of the inverse of $\alpha_1=\rho_M(g_1)$. This illustrates that the von
Neumann algebra of a bilateral stationary Markov process may be ``too small'' to carry a representation of the Thompson group $F$ such that its Markov shift represents the generator $g_0 \in F$.
\end{Remark}

\subsection[Constructions of representations of F from stationary Markov processes]{Constructions of representations of $\boldsymbol F$ from stationary Markov processes} \label{subsection:constr-rep-F}

The following theorem uses the tensor product construction of the present section to show that automorphisms on tensor products give representations of $F$ such that the compressed automorphism is equal to a compressed represented generator.

Throughout this section we will use the following notion of an embedding for two noncommutative probability spaces $(\cA,\varphi)$ and $(\cM,\psi)$. An \emph{embedding} $\iota\colon (\cA,\varphi) \to (\cM,\psi)$ is a~$(\varphi,\psi)$-Markov map $\iota \colon \cA \to \cM$ which is also a $*$-homomorphism. Furthermore, recall the notion
of a~dilation of first order from Definition~\ref{definition:dilation}.

\begin{Theorem}\label{theorem:TensorMarkovF}
Suppose $\gamma \in \Aut(\cA \otimes \cC, \varphi \otimes \chi)$ and let $\iota_0$ be the canonical embedding of $(\cA,\varphi)$ into $(\cA \otimes \cC, \varphi \otimes \chi)$. Then there exists a noncommutative probability space $(\cM,\psi)$, generating representations
$\rho_B, \rho_M \colon F \to \Aut(\cM,\psi)$ and an embedding
$\kappa \colon (\cA \otimes \cC, \varphi \otimes \chi) \to (\cM,\psi)$ such~that
\begin{enumerate}\itemsep=0pt
 \item[$(i)$] $\kappa \iota_0 (\cA) = \cM^{\rho_B(g_0)}$,
 \item[$(ii)$] $\iota_0^*\gamma^n \iota_0 = \iota_0^* \kappa^* \rho_M(g_0^n)\kappa \iota_0$
	for all $n \in \Nset_0$.
\end{enumerate}
In particular, $\big(\cM, \psi, \rho_M(g_0), \cM^{\rho_B(g_0)}\big)$ is
a bilateral noncommutative stationary Markov process.
\end{Theorem}

\begin{proof}
We take
\[
	(\cM, \psi)
	:= \big(\cA \otimes \cC^{\otimes_{\Nset_0^2}}, \varphi \otimes \chi^{\otimes_{\Nset_0^2}}\big)
\]
and let $\kappa$ be the natural
embedding of $(\cA \otimes \cC, \varphi \otimes \chi)$ into $(\cM,\psi)$.
We construct two representations of the Thompson group $F$ as done for the illustrative example in
Section~\ref{subsection:Example}. That is, we define the representation $\rho_B \colon F \to
\Aut(\cM,\psi)$ as $\rho_B(g_n):=\beta_n$ for $n \ge 0$ (see Proposition~\ref{proposition:rep-Fbeta}) and the representation $\rho_M \colon F \to
\Aut(\cM,\psi)$ as $\rho_M(g_n):=\alpha_n$ with $\alpha_0 = \gamma_0 \circ \beta_0$ and $\alpha_n = \beta_n$ for $n \ge 1$ (see Proposition~\ref{proposition:rep-Falpha}).
The generating property of these two representations $\rho_B$ and~$\rho_M$ has already been verified in Propositions~\ref{proposition:rep-Fbeta} and~\ref{proposition:rep-Falpha}.
We recall from Section~\ref{subsection:Example} that $\gamma_0$ is the natural extension of $\gamma$ to an automorphism on $(\cM, \psi)$ which is easily seen to satisfy
\begin{equation} \label{equation:gamma0}
\kappa^*\gamma_0^n \kappa \iota_0 = \gamma^n \iota_0 \qquad \text{for all}\quad n \in \Nset_0.
\end{equation}
Note that for the case $n=1$, the left hand side of this equation can be written as
\begin{equation} \label{equation:gamma0first}
\kappa^*\gamma_0 \kappa \iota_0
=
\kappa^*\gamma_0 \beta_0\kappa \iota_0 =
\kappa^*\alpha_0 \kappa \iota_0.
\end{equation}

Now Proposition~\ref{proposition:MarkovTwoReps2} ensures that $\big(\cM,\psi,\alpha_0,\cM^{\beta_0}\big)$ is a bilateral
noncommutative stationary Markov process with $\kappa \iota_0(\cA)=\cM^{\beta_0}$, as claimed in~$(i)$ of the theorem.
We note that $\kappa \iota_0 (\kappa \iota_0)^*$
is the $\psi$-preserving normal conditional expectation
from $\cM$ onto $\cM^{\beta_0} =
\kappa \iota_0(\cA)$, and by definition, the stationary Markov process
$\big(\cM,\psi,\alpha_0,\cM^{\beta_0}\big)$ has the transition operator
\[
T := \kappa \iota_0 (\kappa \iota_0)^* \alpha_0 \kappa \iota_0 (\kappa \iota_0)^*.
\]

We observe that \eqref{equation:gamma0} and \eqref{equation:gamma0first} allow us to rewrite $T$ as follows:
\begin{align}
T &= \kappa \iota_0 (\kappa \iota_0)^* \alpha_0 \kappa \iota_0 (\kappa \iota_0)^*\nonumber
= \kappa \iota_0\iota_0^*(\kappa^*\alpha_0 \kappa\iota_0)(\kappa\iota_0)^*
\\
&= \kappa\iota_0\iota_0^*(\kappa^*\gamma_0\kappa\iota_0)\iota_0^* \kappa^*
=\kappa\iota_0\iota_0^*\gamma\iota_0\iota_0^*\kappa^*.
\label{equation:Tgamma}
\end{align}

On the other hand, Proposition~\ref{proposition:dilation} gives that
$T$ satisfies
\begin{equation}\label{equation:dilation}
T^n = \kappa \iota_0 (\kappa \iota_0)^* \alpha_0^n \kappa \iota_0 (\kappa \iota_0)^* \qquad
\text{for all}\quad n \in \Nset_0.
\end{equation}

Hence by \eqref{equation:Tgamma} and \eqref{equation:dilation},
\begin{align*}
(\kappa \iota_0\iota_0^*) \gamma^n (\kappa \iota_0\iota_0^*)^*
=[(\kappa \iota_0\iota_0^*) \gamma (\kappa \iota_0\iota_0^*)^*]^n
=T^n= \kappa \iota_0 (\kappa \iota_0)^* \alpha_0^n \kappa \iota_0 (\kappa \iota_0)^*.
\end{align*}

Simplifying, we get
\[
\iota_0^* \gamma^n \iota_0 = \iota_0^* \kappa^* \alpha_0^n \kappa \iota_0 \qquad \text{for all}\quad n \in \Nset_0,
\]
as claimed in $(ii)$ of the theorem.
\end{proof}

This result builds on an observation
related to the existence of Markov dilations already made by K\"ummerer in {\cite[Theorem 4.2.1]{Ku85}}: if a $\varphi$-Markov map $R$ on $\cA$ has a
tensor dilation of first order $(\cA \otimes \cC, \varphi \otimes \chi, \gamma, \iota_0)$, then this implies the existence of a (Markov) dilation on
the noncommutative probability space
$\big(\cA \otimes \cC^{\otimes_{\Zset}^{}}_{}, \varphi \otimes \chi^{\otimes_{\Zset}^{}}_{}\big)$. Here we have utilized this fact and amplified further the dilation to the noncommutative probability space $(\cM,\psi) = \big(\cA \otimes \cC^{\otimes_{\Nset_0^2}} , \varphi \otimes \chi^{\otimes_{\Nset_0^2}}\big)$, such that a~representation of the Thompson group $F$ can be accommodated.

\subsection{The classical case} \label{subsection:constr-classical}

We state a result of K\"ummerer that provides a tensor dilation of any Markov map on a commutative von Neumann algebra. This will allow us to obtain a representation of $F$ as in Theorem~\ref{theorem:TensorMarkovF}.

\begin{Notation} \label{notation:lebesgue}
The (non)commutative probability space $(\cL, \trace_\lambda)$
is given by the Lebesgue space of essentially bounded functions $\cL := L^\infty([0,1],\lambda)$
and $\trace_\lambda := \int_{[0,1]} \cdot\, {\rm d}\lambda$ as the faithful normal state on $\cL$.
Here $\lambda$ denotes the Lebesgue measure on the unit interval $[0,1] \subset \Rset$.
\end{Notation}
\begin{Theorem}[{\cite[4.4.2]{Ku86}}] \label{theorem:kuemmerer-twosided}
Let $R$ be a $\varphi$-Markov map on $\cA$, where $\cA$ is a commutative von
Neumann algebra with separable predual. Then there exists
$\gamma \in \Aut(\cA \otimes \cL, \varphi \otimes \trace_\lambda)$
such that $(\cA\otimes \cL, \varphi\otimes \trace_{\lambda}, \gamma, \iota_0)$ is a Markov $($tensor$)$
dilation of $R$. That is, $(\cA\otimes \cL, \varphi\otimes \trace_{\lambda}, \gamma, \cA\otimes
\1_{\cL})$ is a stationary Markov process, and for all $n \in \Nset_0$,
\[
R^n = \iota_0^* \, \gamma^n \iota_0,
\]
where $\iota_0 \colon (\cA,\varphi) \to (\cA \otimes \cL, \varphi \otimes \trace_\lambda)$
denotes the canonical embedding $\iota_0(a) = a \otimes \1_\cL$ such
that $E_0 := \iota_0 \circ \iota_0^* $ is the $\varphi \otimes \trace_\lambda$-preserving
normal conditional expectation from $\cA \otimes \cL$ onto $\cA \otimes \1_{\cL}$.
\end{Theorem}
A proof of this result on bilateral commutative stationary Markov processes is contained in \cite{Ku86}. For the convenience of the reader, this proof is made available in \cite{KKW20}, with minor modifications to the unilateral setting of such processes.
This folkore result ensures that, in particular, every transition operator of a commutative stationary Markov process has a dilation of first order, which was the starting assumption of Theorem \ref{theorem:TensorMarkovF}.
Consequently, we can associate to
each classical bilateral stationary Markov process a
representation of the
Thompson group~$F$.
\begin{Theorem}\label{theorem:F-gen-compression}
Let $(\cA,\varphi)$ be a noncommutative probability space where $\cA$ is commutative with separable predual,
and let $R$ be a $\varphi$-Markov map on $\cA$. There exists a probability space $(\cM,\psi)$,
generating representations $\rho_B, \rho_M \colon F \to \Aut(\cM,\psi)$, and an embedding
$\iota \colon (\cA, \varphi) \to (\cM,\psi)$ such that
\begin{enumerate}[$(i)$]\itemsep=0pt
\item
$\iota(\cA) = \cM^{\rho_B(g_0)}$,
\item
$R^n = \iota^* \rho_M(g_0^n)\iota $
for all $n \in \Nset_0$. 	
\end{enumerate}	
\end{Theorem}
\begin{proof}
By Theorem \ref{theorem:kuemmerer-twosided}, there exists $\gamma\in \Aut(\cA\otimes
\cL,\varphi\otimes \trace_{\lambda})$ such that $(\cA\otimes \cL, \varphi\otimes
\trace_{\lambda}, \gamma, \cA\otimes \1_{\cL})$ is a stationary Markov process, and $R^n = \iota_0^* \, \gamma^n \iota_0$, for all $n \in \Nset_0$, where $\iota_0 \colon (\cA,\varphi) \to (\cA \otimes \cL, \varphi \otimes \trace_\lambda)$
denotes the canonical embedding $\iota_0(a) = a \otimes \1_\cL$.

By Theorem \ref{theorem:TensorMarkovF},
there exists a probability space $(\cM,\psi)$, generating representations
$\rho_B, \rho_M$: $F \to \Aut(\cM,\psi)$, and an embedding
$\kappa \colon (\cA \otimes \cL, \varphi \otimes \chi) \to (\cM,\psi)$ such that
$\kappa(\cA \otimes \1_{\cL}) = \cM^{\rho_B(g_0)}$ and $\iota_0^* \gamma^n \iota_0 = \iota_0^* \kappa^* \rho_M(g_0^n)\kappa \iota_0$ for all $n \in \Nset_0$. The proof is completed by taking $\iota : = \kappa \circ \iota_0$, as we get
\begin{align*}
 R^n = \iota_0^* \gamma^n \iota_0
 = \iota_0^* \kappa^* \rho_M(g_0^n)\kappa \iota_0 = \iota^* \rho_M(g_0^n)\iota \qquad \text{for all}\quad n \in \Nset_0.\tag*{\qed}
\end{align*}
\renewcommand{\qed}{}
\end{proof}

\subsection{Further discussion of the classical case} \label{subsection:discuss-classical}
We illustrate Theorem \ref {theorem:F-gen-compression} for
a classical stationary Markov process taking values in the finite set $[d] := \{1,2, \ldots, d\}$ for some $d \ge 2$,
adapting the classical construction of such processes to our
algebraic approach.

Consider the unital *-algebra
$\cA := \Cset^d \cong
\{f \colon [d] \to \Cset\}$. Then
$\varphi(f) := \sum_{i=1}^d q_i f(i)$
defines a~faithful (normal tracial) state $\varphi$ on $\cA$ if and only if
$\sum_{i=1}^d q_i =1$ and $0 < q_i < 1$ for all $1 \le i \le d$. Now consider
the transition operator $R \colon \cA \to \cA$ given by the matrix
\[
R = \begin{bmatrix}
p_{1,1} & p_{1,2} & \cdots & p_{1,d} \\
p_{2,1} & p_{2,2} & \cdots & p_{2,d} \\
\vdots & \vdots & \ddots & \vdots \\
p_{d,1} & p_{d,2} & \cdots & p_{d,d} \\
\end{bmatrix}
\]
for some $p_{i,j} \in [0,1]$ satisfying
$\sum_{j=1}^d p_{i,j} =1$ for all $i=1, \ldots, d$. One easily verifies that
\begin{align*}
\varphi \circ R = \varphi
\, \Longleftrightarrow \,
\sum_{i=1}^{d} q_i p_{i,j} = q_j \qquad \text{for all}\quad 1 \le j \le d \qquad \text{(stationarity)}.
\end{align*}
The usual Daniell--Kolmogorov construction of a stationary Markov process can now be algebraically reformulated as follows. Here we closely follow the exposition provided in \cite{Ku86}. A~state~$\widetilde{\varphi}$ is defined on the infinite algebraic
tensor product $\bigodot_\Zset \cA$ by
\begin{gather*}
\widetilde{\varphi}
(\cdots \otimes \1_\cA \otimes f_{-m} \otimes
f_{-m+1} \otimes \cdots \otimes f_{n-1} \otimes f_{n} \otimes \1_\cA \otimes \cdots)
\\ \qquad
{}:= \varphi\big(f_{-m} R(f_{-m+1} R(\cdots f_{n-1} R(f_n) \cdots )) \big).
\end{gather*}
This state $\widetilde{\varphi}$ extends to a faithful normal state $\widehat{\varphi}$ on the
von Neumann algebraic tensor product
$\widehat{\cA } :=\bigotimes_\Zset \cA$ such that $\big(\widehat{\cA}, \widehat{\varphi}\big)$ is a noncommutative probability space
(in the sense of Section~\ref{subsection:Markov-maps}).
Furthermore,
the tensor right shift on $\bigodot_\Zset \cA$ extends to an
automorphism $\widehat{\cA}$ of
$\big(\widehat{\cA},\widehat{\varphi}\big)$.
Finally, let $\widehat{\iota}_{DK}\colon \cA \to \widehat{\cA}$ denote the injection
which canonically embeds $f \in \cA$ into the $0$-th position of the infinite tensor product $\widehat{\cA } =\bigotimes_\Zset \cA$.
Then it can be
verified that $\big(\widehat{\cA}, \widehat{\varphi}, \widehat{T},\widehat{\iota}_{DK}(\cA)\big)$ is a~minimal stationary Markov process (in
the sense of Definition~\ref{definition:ncms}).

However, the Daniell--Kolmogorov construction does not seem to accommodate
a representation $\widehat{\rho} \colon F \to \Aut\big(\widehat{A}, \widehat{\varphi}\big)$ with $\widehat{\rho}(g_0) =
\widehat{T}$ which satisfies the additional
localization property $\widehat{\iota}_{DK}(\cA) \subset \widehat{\cA}^{\widehat{\rho}(g_n)}$ for $n \ge 1$.
This observation is connected to the well-known fact that the Daniell--Kolmogorov construction puts all information about a stochastic process into the state~$\widehat{\varphi}$, while the
automorphism $\widehat{T}$ is simply implemented by a bilateral tensor shift.

 Fortunately, K\"ummerer's approach to the construction of stationary Markov processes is more feasible for finding representations of the Thompson group $F$
 with properties as addressed above. This open dynamical system approach is alternative to the Daniell--Kolmogorov construction in classical probability;
 and it is actually independent
 of it for finite-set-valued processes. As explained in~\cite{Ku86}, this alternative approach
 provides a construction which puts some information of the stationary Markov process into the automorphism while simplifying the state
 (see Theorem~\ref{theorem:kuemmerer-twosided}).
 More specifically, this
 strategy divides the construction into two steps.
 One first tries to construct a dilation of first order, and then one attempts in a second step to extend this first-order dilation to a full (Markov) dilation (see Section~\ref{subsection:Noncommutative Stationary Processes}). In fact,
 as already observed in Section~\ref{subsection:constr-rep-F},
 this two-step strategy can be further extended to construct a representation of the Thompson group $F$ which encodes the Markovianity of the given stationary process.
 Let us further discuss this alternative construction for a tensor dilation for the present example $\big(\cA = \Cset^d, \varphi\big)$ with transition operator $R$ on $\cA$. For
 this purpose,
 recall Notation \ref{notation:lebesgue}.
 Similar as done for the case $d=2$ in \cite[Example~3.4.3]{KKW20}
 and as detailed in~\cite{Ku86}, one can construct an automorphism $\gamma \in
\Aut(\cA \otimes \cL, \varphi \otimes \trace_\lambda)$ such that the $\varphi$-Markov map~$R$
on $\cA$ has the dilation of first order $(\cA \otimes \cL, \varphi \otimes \trace_{\lambda}, \gamma, \iota_0)$. As before, $\iota_0$~denotes the canonical embedding of $(\cA, \varphi)$ into $(\cA \otimes \cL, \varphi \otimes \trace_{\lambda})$. In other words, the diagram\looseness=1
	\begin{equation} \label{eq:cd-first-order}
		\begin{tikzcd}
	(\cA, \varphi) \arrow[r, "R"] \arrow[d, "\iota_0"]
	& (\cA, \varphi) \arrow[d, leftarrow, "\iota_0^*"] \\
	(\cA \otimes \cL, \varphi \otimes \trace_{\lambda}) \arrow[r, "\gamma"]
	& (\cA \otimes \cL, \varphi \otimes \trace_{\lambda})
	\end{tikzcd}
	\end{equation}
commutes.

\begin{Remark}
All information about the $\varphi$-Markov map $R$ on $\cA$ is contained in the $\varphi \otimes \trace_\lambda$-preserving automorphism $\gamma$ on $\cA \otimes \cL$. Generally, $\cA_\Zset := \bigvee_{n\in \Zset}\gamma^n(\cA \otimes \1_{\cL})$ is
strictly contained in $\cA \otimes \cL$. In other words, Theorem \ref{theorem:kuemmerer-twosided} provides a non-minimal stationary Markov process, in general. Actually, our first step in the construction of a representation of the Thompson group~$F$ consists in finding a suitable dilation of first order \eqref{eq:cd-first-order}. K\"ummerer's Theorem \ref{theorem:kuemmerer-twosided} guarantees the existence of such dilations. However, we refrain from further discussing the structure of these dilations of first order, as this would go beyond the scope of the present paper.\looseness=1
\end{Remark}

Having arrived at this dilation of first order, several straightforward constructions of stationary Markov processes are possible. Here we discuss those which are of relevance for obtaining unilateral and bilateral versions of stationary Markov processes, in particular with the view of obtaining suitable representations of the Thompson group $F$, and its monoid $F^+$, as introduced in \eqref{eq:F+}.

A unilateral noncommutative stationary Markov process
$\big(\tcM,\tpsi,\talpha_0,\tiota(\cA)\big)$ is
obtained by putting
$\big(\tcM, \tpsi\big) := \big(\cA \otimes \cL^{\otimes_{\Nset_0}^{}}_{},
\varphi \otimes \trace_\lambda^{\otimes_{\Nset_0}^{}}\big)
$ with $\talpha_0 := \tgamma_0 \tbeta_0$, where
\begin{gather*}
\tbeta_0(f \otimes x_0 \otimes x_1 \otimes \cdots) := f \otimes \1_\cL \otimes x_0 \otimes x_1 \otimes \cdots, \\
\tgamma_0(f \otimes x_0 \otimes x_1 \otimes \cdots) := \gamma(f \otimes x_0) \otimes x_1 \otimes \cdots, \\
\tiota(f) := f \otimes \1_\cL \otimes \1_{\cL} \otimes \cdots
\end{gather*}
for $f \in \cA$, $x_0, x_1, \ldots \in \cL$. This
construction was the subject of \cite{KKW20}, as it
allows to introduce the representations $\trho_B$ and $\trho_M$ of the Thompson monoid $F^+$ by putting
\begin{align}
\trho_B(g_k) &:= \tbeta_k \qquad\text{for}\quad k\ge 0, \label{eq:ex-uni-bernoulli}\\
\trho_M(g_k) &:= \begin{cases}
 \talpha_0 & \text{for}\ k=0,
 \\
 \tbeta_k & \text{for}\ k>0,
 \end{cases} \label{eq:ex-uni-markov}
\end{align}
with
$\tbeta_k(f \otimes x_0 \otimes \cdots \otimes x_{k-1} \otimes x_k \otimes x_{k+1} \otimes \cdots) := f \otimes x_0 \otimes \cdots \otimes x_{k-1} \otimes \1_{\cL} \otimes x_{k} \otimes \cdots$. It is now elementary to verify the relations
\begin{gather}
 \tbeta_k \tbeta_\ell = \tbeta_{\ell+1} \tbeta_k, \qquad 0 \le k \le \ell < \infty,\nonumber
\\
 \talpha_k \talpha_\ell = \talpha_{\ell+1} \talpha_k, \qquad 0 \le k < \ell < \infty. \label{eq:beta-end}
\end{gather}
The choices made in \eqref{eq:ex-uni-bernoulli}
are canonical for the partial shifts $\tbeta_k$ (see also \cite{EGK17,KKW20}).
The choice made in \eqref{eq:ex-uni-markov} is
also canonical from the dynamical systems viewpoint of constructing a stationary Markov process as a
local perturbation of a Bernoulli shift. But of course, other choices are possible for $\trho_M(g_k)$ for $k\ge 1$, respecting the
localization property $\tiota(\cA) \subset \tcM^{\rho_M(g_k)}$, without violating the
relations of the Thompson monoid $F^+$
 (see also \cite[Section~5.3]{KKW20}).
This construction is nicely illustrated in Figure \ref{figure:graphs-onesided} with actions of injective maps on the set $\{\blacksquare\}\sqcup \Nset_0$. Here the set~$\{\blacksquare\}$ pictures the algebra $\cA$ (or an element of it), $\bullet$ pictures a~copy of the algebra $\cL$ (or an element of it), and disjoint unions of sets correspond to tensor products in the algebraic formulation. Now the action of the partial shifts $\tbeta_0$ and $\tbeta_1$ become injective maps on the set
$\{\blacksquare\}\sqcup \Nset_0$ which can be visualized by blue arrows. Furthermore, the
action of the local automorphism $\tgamma_0$ is visualized by a bijection on $\{\blacksquare\}\sqcup \Nset_0$ which moves only those elements inside the red ellipse, as indicated in red colour in Figure \ref{figure:graphs-onesided}. A similar visualization is immediate for the actions of $\tbeta_k$ for $k >1$. We finally note for
Figure \ref{figure:graphs-onesided} that
$\circ$ visualizes the one-dimensional subalgebra
$\Cset \1_\cL \subset \cL$ (or its element $\1_\cL$) which is actually given by the empty set $\varnothing$ on the level of sets. Here we could have omitted these
isomorphic embeddings for our visualization, but these embeddings will guide our consecutive amplifications, in particular as relevant for canonically constructing representations of $F$. As it can be clearly seen in Figure \ref{figure:graphs-onesided}, the set $\{\blacksquare\}\sqcup \Nset_0$ is invariant for the injections which visualize the actions of $\tilde{\beta}_k-s$ and $\tgamma_0$.
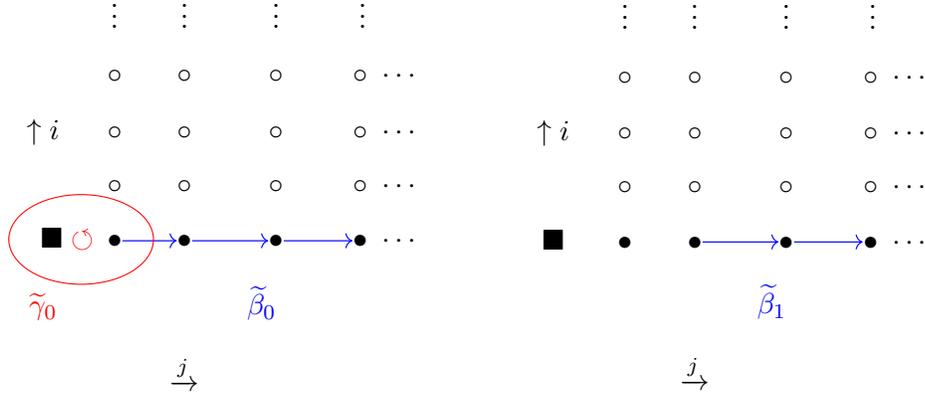
\begin{figure}[t!]\centering
\begin{tikzcd}[column sep=0.8em, row sep=0.5em,execute at end picture={
 \node[mynode, fit={(\tikzcdmatrixname-6-1) (\tikzcdmatrixname-6-2)}]{};
 }]
& \vdots & \vdots& \vdots & \vdots \phantom{\,\,\cdots} \\[-3pt]
&
\\
& \circ
& \circ
& \circ
& \circ \,\, \cdots \\
\uparrow{i} & \circ
& \circ
& \circ
& \circ \,\, \cdots \\
 & \circ
& \circ
& \circ
& \circ \,\, \cdots\\
\blacksquare \color{red} \hspace{0pt} \circlearrowleft \hspace{-18pt}
& \bullet \arrow[r, blue, start anchor={[xshift=-1ex]}, end anchor={[xshift=1ex]}]
& \bullet \arrow[r, blue, start anchor={[xshift=-1ex]}, end anchor={[xshift=1ex]}]
& \bullet \arrow[r, blue, start anchor={[xshift=-1ex]}, end anchor={[xshift=1ex]}]
& \bullet \,\, \cdots \\
 \color{red} \tgamma_0 & & & {\color{blue} \tbeta_0 \quad} & \\
 & & \xrightarrow{j} & &
\end{tikzcd} \qquad \quad
\begin{tikzcd}[column sep=0.8em, row sep=0.5em]
& \vdots & \vdots& \vdots & \vdots \phantom{\,\,\cdots} \\[-3pt]
&
\\
& \circ
& \circ
& \circ
& \circ \,\, \cdots \\
\uparrow{i} & \circ
& \circ
& \circ
& \circ \,\, \cdots \\
 & \circ
& \circ
& \circ
& \circ \,\, \cdots\\
\blacksquare
& \bullet
& \bullet \arrow[r, blue, start anchor={[xshift=-1ex]}, end anchor={[xshift=1ex]}]
& \bullet \arrow[r, blue, start anchor={[xshift=-1ex]}, end anchor={[xshift=1ex]}]
& \bullet \,\, \cdots \\
 & & & {\color{blue} \tbeta_1 \quad} & \\
 & & \xrightarrow{j} & &
\end{tikzcd}
\caption{Visualization
on the set $\{\blacksquare\}\sqcup \Nset_0$ of the action of the one-sided Bernoulli shift $\tbeta_0$ (blue, left), and the local automorphism $\tgamma_0$ (red, left) and the action of the one-sided Bernoulli shift $\tbeta_1$ (blue, right).}
\label{figure:graphs-onesided}
\end{figure}

Next, we extend the unilateral stationary Markov process $\big(\tcM,\tpsi,\talpha_0,\tiota(\cA)\big)$ to the
bilateral stationary Markov process $\big(\hcM,\hpsi,\halpha_0,\hiota(\cA)\big)$
by putting $(\hcM, \hpsi) := \big(\cA \otimes \cL^{\otimes_{\Zset}^{}}_{},
\varphi \otimes \trace_\lambda^{\otimes_{\Zset}^{}}\big)$ with $\halpha_0 := \hgamma_0 \hbeta_0$, where
\begin{gather*}
\hbeta_0\left(\cdots \otimes x_{-1} \otimes
\begin{pmatrix} x_{0}\\ \otimes\\ f\end{pmatrix}
\otimes x_1 \otimes \cdots\right) := \cdots \otimes x_{-2} \otimes
\begin{pmatrix} x_{-1}\\ \otimes\\ f\end{pmatrix}
\otimes x_0 \otimes \cdots, \\
\hgamma_0\left(\cdots \otimes x_{-1} \otimes
\begin{pmatrix} x_{0}\\ \otimes\\ f\end{pmatrix}
\otimes x_1 \otimes \cdots \right) := \cdots \otimes x_{-1} \otimes
 \gamma_0 \begin{pmatrix} x_{0}\\ \otimes\\ f\end{pmatrix}
\otimes x_1 \otimes \cdots,\\
\hiota(f) := \cdots \otimes \1_{\cL} \otimes
\begin{pmatrix} \1_{\cL}\\ \otimes\\ f\end{pmatrix}
\otimes \1_{\cL} \otimes \cdots
\end{gather*}
for $f \in \cA$, $\ldots, x_{-1},x_0, x_1, \ldots \in \cL$.
Considering the automorphism $\halpha_0$
as a canonical bilateral extension of the endomorphism $\talpha_0$, we are interested in identifying bilateral extensions of the other endomorphisms $\talpha_1, \talpha_2, \ldots$ to automorphisms of
$\big(\hcM,\hpsi\,\big)$, now satisfying the relations of the Thompson group~$F$.
But this seems to be impossible, as
$\big(\hcM,\hpsi\,\big)$ provides ``too little space'' for accommodating such automorphisms. This is illustrated in Figure~\ref{figure:graphs-twosided} again on the level of the set $\{\blacksquare\} \sqcup \Zset$, when visualized as an appropriate
subset of $\{\blacksquare\} \sqcup \Nset_0^2$.
Note that we have made a particular choice of how to embed $\{\blacksquare\} \sqcup \Zset$ into $\{\blacksquare\} \sqcup \Nset_0^2$, and there are many other interesting possibilities for choosing such an embedding.
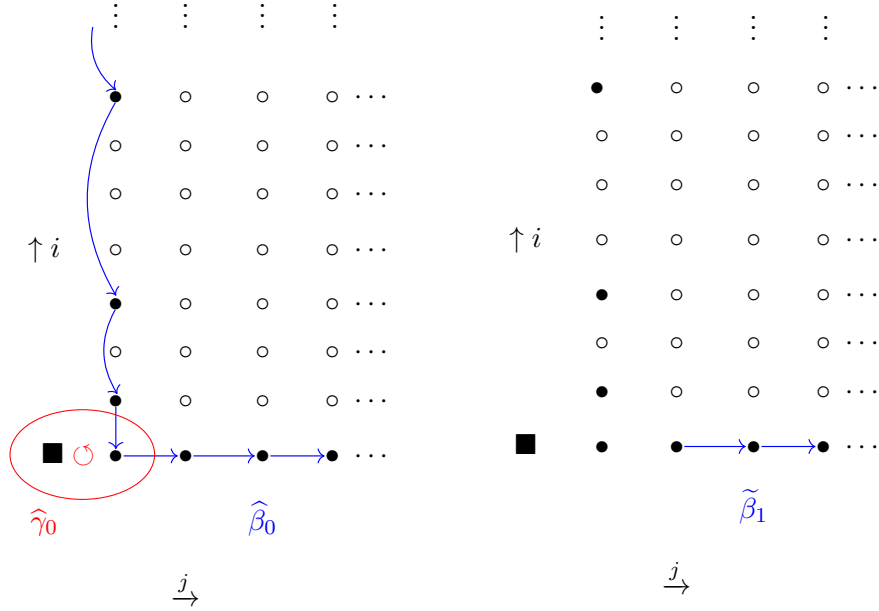
\begin{figure}[t!]\centering
\begin{tikzcd}[column sep=0.8em, row sep=0.5em, execute at end picture={
 \node[mynode, fit={(\tikzcdmatrixname-10-1) (\tikzcdmatrixname-10-2)}]{};
 }]
& \vdots & \vdots& \vdots & \vdots \phantom{\,\,\cdots} \\[-3pt]
& \arrow[d, blue, start anchor={[xshift=-1ex,yshift=3.2ex]}, end anchor={[xshift=0.8ex,yshift=-0.8ex]}, bend right] & & & \\
& \bullet \arrow[dddd, blue, start anchor={[xshift=0.8ex,yshift=1ex]}, end anchor={[xshift=0.8ex, yshift=-0.7ex]}, bend right]
& \circ
& \circ
& \circ \,\, \cdots \\
& \circ
& \circ
& \circ
& \circ \,\, \cdots \\
& \circ
& \circ
& \circ
& \circ \,\, \cdots \\
\uparrow{i} & \circ
& \circ
& \circ
& \circ \,\, \cdots \\
 & \bullet \arrow[dd, blue, start anchor={[xshift=0.8ex,yshift=1ex]}, end anchor={[xshift=0.9ex,yshift=-0.8ex]}, bend right]
& \circ
& \circ
& \circ \,\, \cdots \\
 & \circ
& \circ
& \circ
& \circ \,\, \cdots \\
 & \bullet \arrow[d, blue, start anchor={[yshift=1ex]}, end anchor={[yshift=-0.8ex]}]
& \circ
& \circ
& \circ \,\, \cdots\\
\blacksquare \color{red} \hspace{0pt} \circlearrowleft \hspace{-18pt}
& \bullet \arrow[r, blue, start anchor={[xshift=-1ex]}, end anchor={[xshift=1ex]}]
& \bullet \arrow[r, blue, start anchor={[xshift=-1ex]}, end anchor={[xshift=1ex]}]
& \bullet \arrow[r, blue, start anchor={[xshift=-1ex]}, end anchor={[xshift=1ex]}]
& \bullet \,\, \cdots \\
 \color{red} \hgamma_0 & & & \color{blue} \hbeta_0& \\
 & & \xrightarrow{j} &&
\end{tikzcd} \qquad \quad
\begin{tikzcd}[column sep=0.8em, row sep=0.5em]
& \vdots & \vdots& \vdots & \vdots \phantom{\,\,\cdots} \\[-3pt]
& & & & \\
& \bullet \
& \circ
& \circ
& \circ \,\, \cdots \\
& \circ
& \circ
& \circ
& \circ \,\, \cdots \\
& \circ
& \circ
& \circ
& \circ \,\, \cdots \\
\uparrow{i} & \circ
& \circ
& \circ
& \circ \,\, \cdots \\
& \bullet
& \circ
& \circ
& \circ \,\, \cdots \\
& \circ
& \circ
& \circ
& \circ \,\, \cdots \\
 & \bullet
& \circ
& \circ
& \circ \,\, \cdots\\
\blacksquare
& \bullet
& \bullet \arrow[r, blue, start anchor={[xshift=-1ex]}, end anchor={[xshift=1ex]}]
& \bullet \arrow[r, blue, start anchor={[xshift=-1ex]}, end anchor={[xshift=1ex]}]
& \bullet \,\, \cdots \\
 & & & \color{blue} \tbeta_1& \\
 & & \xrightarrow{j} &&
\end{tikzcd}
\caption{Visualization on the set $\{\blacksquare\} \sqcup \Zset$ of the action of the two-sided Bernoulli shift $\hbeta_0$ and the
local automorphism $\hgamma_0$ and of the ``inability'' to extend $\tbeta_1$
from $\{\blacksquare\} \sqcup \Nset_0$ to an automorphism $\hbeta_1$ on $\{\blacksquare\} \sqcup \Zset$ such that the relations of $F$ are satisfied.}
\label{figure:graphs-twosided}
\end{figure}
This challenge to provide sufficient space for properly extending all partial shifts
$\big\{\tbeta_k \mid k \ge 0\big\} \subset \big(\tcM,\tvarphi\big)$ is overcome by choosing
\[
(\mathcal{M}, \psi) = \big(\mathcal{A}\otimes \mathcal{L}^{\otimes_{\mathbb{N}_0^2}}, \varphi \otimes \trace_{\lambda}^{\otimes_{\mathbb{N}_0^2}} \big)
\]
with the canonical embedding
$\iota \colon (\cA,\varphi) \to (\cM, \psi)$ given by
$\iota(a) := a \otimes \big(\bigotimes_{(i,j)\in \Nset_0^2} \1_{\cL}\big)$.
This approach has already been detailed in the illustrative example of Section~\ref{subsection:Example}. For the convenience of the reader, let us repeat how
the partial shifts $\tbeta_k$ and the local automorphism $\tgamma_0$ on $\tcM$ are extended to automorphisms on $\cM$:
\[
\beta_0\biggl(a \otimes \biggl(\bigotimes_{(i,j)\in \Nset_0^2} x_{i,j}\biggr)\biggr):= a \otimes \biggl(\bigotimes_{(i,j)\in \Nset_0^2} y_{i,j}\biggr)\qquad
\text{with}\quad
y_{i,j} = \begin{cases}
x_{2i+1,j} & \text{if}\ j=0, \\
x_{2i,j-1} & \text{if}\ j=1, \\
x_{i,j-1} & \text{if}\ j \geq 2,
\end{cases}
\]
and, for $k \in \Nset$,
\[
\beta_k\biggl(a \otimes \biggl(\bigotimes_{(i,j)\in \Nset_0^2} x_{i,j}\biggr)\biggr):= a \otimes \biggr(\bigotimes_{(i,j)\in \Nset_0^2} y_{i,j}\biggr)
\qquad\text{with}\quad
y_{i,j} = \begin{cases}
x_{i,j} & \text{if}\ j\le k-1, \\
x_{2i+1,j} & \text{if}\ j=k, \\
x_{2i,j-1} & \text{if}\ j=k+1, \\
x_{i,j-1} & \text{if}\ j \geq k+1.
\end{cases}
\]
Furthermore, the local perturbation $\gamma \in \Aut(\cA, \cL)$ is amplified to
\[
\gamma_0\biggl(a \otimes \biggl(\bigotimes_{(i,j)\in \Nset_0^2} x_{i,j}\biggr)\biggr) = \gamma(a \otimes x_{00}) \otimes \biggl(\bigotimes_{(i,j)\in \Nset_0^2 \setminus \{(0,0)\}} x_{i,j}\biggr).
\]
 We refer the reader to Figure \ref{figure:graphs-twodim} for a
visualization of the action of the two-sided shifts $\beta_0$, $\beta_1$ and the action of the local automorphism $\gamma_0$.

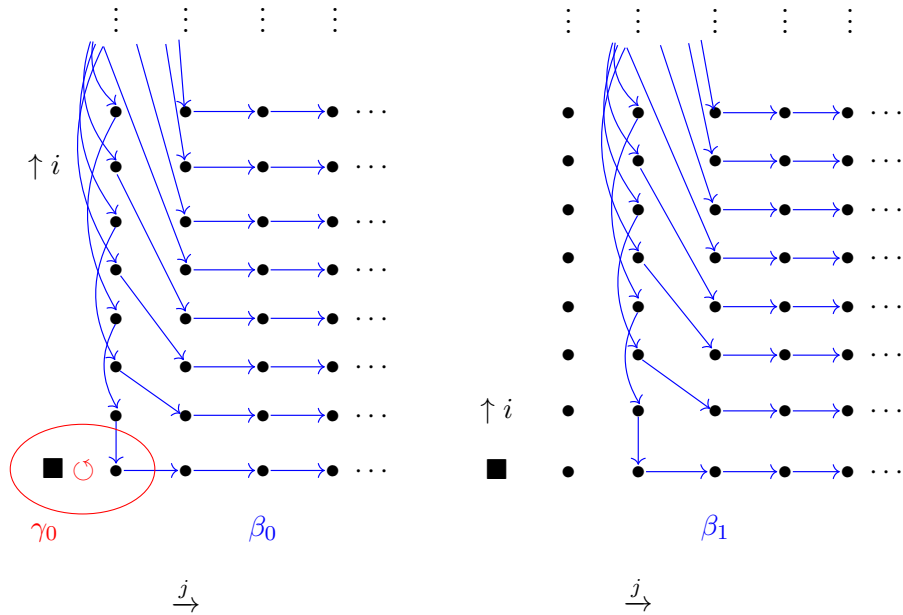
\begin{figure}[h]\centering
\begin{tikzcd}[column sep=0.8em, row sep=0.5em, execute at end picture={
 \node[mynode, fit={(\tikzcdmatrixname-10-1) (\tikzcdmatrixname-10-2)}]{};
 }]
& \vdots & \vdots& \vdots & \vdots \phantom{\,\,\cdots} \\[1pt]
& \arrow[ddddr, blue, start anchor={[xshift=-1.5ex,yshift=2.8ex]}, end anchor={[xshift=0.5ex,yshift=-0.8ex]}] \arrow[dddd, blue, start anchor={[xshift=-0.2ex,yshift=2.8ex]}, end anchor={[xshift=0.8ex,yshift=-0.8ex]}, bend right] \arrow[dddr, blue, start anchor={[xshift=1ex,yshift=3ex]}, end anchor={[xshift=0.5ex,yshift=-0.8ex]}] \arrow[ddd, blue, start anchor={[xshift=-0.8ex,yshift=3ex]}, end anchor={[xshift=0.8ex,yshift=-0.8ex]}, bend right] \arrow[ddr, blue, start anchor={[xshift=3ex,yshift=3ex]}, end anchor={[xshift=0.8ex,yshift=-0.8ex]}]
\arrow[dd, blue, start anchor={[xshift=-1.2ex,yshift=3.2ex]}, end anchor={[xshift=0.8ex,yshift=-0.8ex]}, bend right]
\arrow[d, blue, start anchor={[xshift=-1ex,yshift=3.2ex]}, end anchor={[xshift=0.8ex,yshift=-0.8ex]}, bend right] \arrow[dr, blue, start anchor={[xshift=4ex,yshift=2.6ex]}, end anchor={[xshift=1.5ex,yshift=-0.8ex]}] & & & \\
 & \bullet \arrow[dddd, blue, start anchor={[xshift=0.8ex,yshift=1ex]}, end anchor={[xshift=0.8ex, yshift=-0.7ex]}, bend right]
& \bullet \arrow[r, blue, start anchor={[xshift=-1ex]}, end anchor={[xshift=1ex]}]
& \bullet \arrow[r, blue, start anchor={[xshift=-1ex]}, end anchor={[xshift=1ex]}]
& \bullet \,\, \cdots \\
\uparrow{i} & \bullet \arrow[dddr, blue, start anchor={[xshift=-0.6ex,yshift=0.9ex]}, end anchor={[xshift=0.6ex,yshift=-0.8ex]}]
& \bullet \arrow[r, blue, start anchor={[xshift=-1ex]}, end anchor={[xshift=1ex]}]
& \bullet \arrow[r, blue, start anchor={[xshift=-1ex]}, end anchor={[xshift=1ex]}]
& \bullet \,\, \cdots \\
& \bullet \arrow[ddd, blue, start anchor={[xshift=0.8ex, yshift=1ex]}, end anchor={[xshift=0.7ex,yshift=-0.8ex]}, bend right]
& \bullet \arrow[r, blue, start anchor={[xshift=-1ex]}, end anchor={[xshift=1ex]}]
& \bullet \arrow[r, blue, start anchor={[xshift=-1ex]}, end anchor={[xshift=1ex]}]
& \bullet \,\, \cdots \\
& \bullet \arrow[ddr, blue, start anchor={[xshift=-0.7ex,yshift=1ex]}, end anchor={[xshift=0.8ex,yshift=-0.8ex]}]
& \bullet \arrow[r, blue, start anchor={[xshift=-1ex]}, end anchor={[xshift=1ex]}]
& \bullet \arrow[r, blue, start anchor={[xshift=-1ex]}, end anchor={[xshift=1ex]}]
& \bullet \,\, \cdots \\
& \bullet \arrow[dd, blue, start anchor={[xshift=0.8ex,yshift=1ex]}, end anchor={[xshift=0.9ex,yshift=-0.8ex]}, bend right]
& \bullet \arrow[r, blue, start anchor={[xshift=-1ex]}, end anchor={[xshift=1ex]}]
& \bullet \arrow[r, blue, start anchor={[xshift=-1ex]}, end anchor={[xshift=1ex]}]
& \bullet \,\, \cdots \\
& \bullet \arrow[dr, blue, start anchor={[xshift=-1.3ex,yshift=0.8ex]}, end anchor={[xshift=1.2ex,yshift=-1.0ex]}]
& \bullet \arrow[r, blue, start anchor={[xshift=-1ex]}, end anchor={[xshift=1ex]}]
& \bullet \arrow[r, blue, start anchor={[xshift=-1ex]}, end anchor={[xshift=1ex]}]
& \bullet \,\, \cdots \\
 & \bullet \arrow[d, blue, start anchor={[yshift=1ex]}, end anchor={[yshift=-0.8ex]}]
& \bullet \arrow[r, blue, start anchor={[xshift=-1ex]}, end anchor={[xshift=1ex]}]
& \bullet \arrow[r, blue, start anchor={[xshift=-1ex]}, end anchor={[xshift=1ex]}]
& \bullet \,\, \cdots\\
\blacksquare \color{red} \hspace{0pt} \circlearrowleft \hspace{-18pt}
& \bullet \arrow[r, blue, start anchor={[xshift=-1ex]}, end anchor={[xshift=1ex]}]
& \bullet \arrow[r, blue, start anchor={[xshift=-1ex]}, end anchor={[xshift=1ex]}]
& \bullet \arrow[r, blue, start anchor={[xshift=-1ex]}, end anchor={[xshift=1ex]}]
& \bullet \,\, \cdots \\
 \color{red} \gamma_0 & & & \color{blue} \beta_0& \\
 & & \xrightarrow{j} &&
\end{tikzcd} \quad \quad
\begin{tikzcd}[column sep=0.8em, row sep=0.5em]
& \vdots & \vdots & \vdots& \vdots & \vdots \phantom{\,\,\cdots} \\[1pt]
& & \arrow[ddddr, blue, start anchor={[xshift=-1.5ex,yshift=2.8ex]}, end anchor={[xshift=0.5ex,yshift=-0.8ex]}] \arrow[dddd, blue, start anchor={[xshift=-0.2ex,yshift=2.8ex]}, end anchor={[xshift=0.8ex,yshift=-0.8ex]}, bend right] \arrow[dddr, blue, start anchor={[xshift=1ex,yshift=3ex]}, end anchor={[xshift=0.5ex,yshift=-0.8ex]}] \arrow[ddd, blue, start anchor={[xshift=-0.8ex,yshift=3ex]}, end anchor={[xshift=0.8ex,yshift=-0.8ex]}, bend right] \arrow[ddr, blue, start anchor={[xshift=3ex,yshift=3ex]}, end anchor={[xshift=0.8ex,yshift=-0.8ex]}]
\arrow[dd, blue, start anchor={[xshift=-1.2ex,yshift=3.2ex]}, end anchor={[xshift=0.8ex,yshift=-0.8ex]}, bend right]
\arrow[d, blue, start anchor={[xshift=-1ex,yshift=3.2ex]}, end anchor={[xshift=0.8ex,yshift=-0.8ex]}, bend right] \arrow[dr, blue, start anchor={[xshift=4ex,yshift=2.6ex]}, end anchor={[xshift=1.5ex,yshift=-0.8ex]}] & & & \\
&\bullet & \bullet \arrow[dddd, blue, start anchor={[xshift=0.8ex,yshift=1ex]}, end anchor={[xshift=0.8ex, yshift=-0.7ex]}, bend right]
& \bullet \arrow[r, blue, start anchor={[xshift=-1ex]}, end anchor={[xshift=1ex]}]
& \bullet \arrow[r, blue, start anchor={[xshift=-1ex]}, end anchor={[xshift=1ex]}]
& \bullet \,\, \cdots \\
&\bullet & \bullet \arrow[dddr, blue, start anchor={[xshift=-0.6ex,yshift=0.9ex]}, end anchor={[xshift=0.6ex,yshift=-0.8ex]}]
& \bullet \arrow[r, blue, start anchor={[xshift=-1ex]}, end anchor={[xshift=1ex]}]
& \bullet \arrow[r, blue, start anchor={[xshift=-1ex]}, end anchor={[xshift=1ex]}]
& \bullet \,\, \cdots \\
&\bullet & \bullet \arrow[ddd, blue, start anchor={[xshift=0.8ex, yshift=1ex]}, end anchor={[xshift=0.7ex,yshift=-0.8ex]}, bend right]
& \bullet \arrow[r, blue, start anchor={[xshift=-1ex]}, end anchor={[xshift=1ex]}]
& \bullet \arrow[r, blue, start anchor={[xshift=-1ex]}, end anchor={[xshift=1ex]}]
& \bullet \,\, \cdots \\
&\bullet & \bullet \arrow[ddr, blue, start anchor={[xshift=-0.7ex,yshift=1ex]}, end anchor={[xshift=0.8ex,yshift=-0.8ex]}]
& \bullet \arrow[r, blue, start anchor={[xshift=-1ex]}, end anchor={[xshift=1ex]}]
& \bullet \arrow[r, blue, start anchor={[xshift=-1ex]}, end anchor={[xshift=1ex]}]
& \bullet \,\, \cdots \\
&\bullet & \bullet \arrow[dd, blue, start anchor={[xshift=0.8ex,yshift=1ex]}, end anchor={[xshift=0.9ex,yshift=-0.8ex]}, bend right]
& \bullet \arrow[r, blue, start anchor={[xshift=-1ex]}, end anchor={[xshift=1ex]}]
& \bullet \arrow[r, blue, start anchor={[xshift=-1ex]}, end anchor={[xshift=1ex]}]
& \bullet \,\, \cdots \\
&\bullet & \bullet \arrow[dr, blue, start anchor={[xshift=-1.3ex,yshift=0.8ex]}, end anchor={[xshift=1.2ex,yshift=-1.0ex]}]
& \bullet \arrow[r, blue, start anchor={[xshift=-1ex]}, end anchor={[xshift=1ex]}]
& \bullet \arrow[r, blue, start anchor={[xshift=-1ex]}, end anchor={[xshift=1ex]}]
& \bullet \,\, \cdots \\
\uparrow{i} &\bullet & \bullet \arrow[d, blue, start anchor={[yshift=1ex]}, end anchor={[yshift=-0.8ex]}]
& \bullet \arrow[r, blue, start anchor={[xshift=-1ex]}, end anchor={[xshift=1ex]}]
& \bullet \arrow[r, blue, start anchor={[xshift=-1ex]}, end anchor={[xshift=1ex]}]
& \bullet \,\, \cdots\\
\blacksquare
&\bullet & \bullet \arrow[r, blue, start anchor={[xshift=-1ex]}, end anchor={[xshift=1ex]}]
& \bullet \arrow[r, blue, start anchor={[xshift=-1ex]}, end anchor={[xshift=1ex]}]
& \bullet \arrow[r, blue, start anchor={[xshift=-1ex]}, end anchor={[xshift=1ex]}]
& \bullet \,\, \cdots \\
 & & & \color{blue} \beta_1 & \\
 & &\xrightarrow{j} &&
\end{tikzcd} \qquad \quad
\caption{Visualization on the set $\{\blacksquare\} \sqcup \Nset_0^2$ of the action of the two-sided Bernoulli shift $\beta_0$, the local automorphism $\gamma_0$, and the two-sided Bernoulli shift $\beta_1$.}
\label{figure:graphs-twodim}
\end{figure}

We address $\{\beta_k \mid k \ge 0\}$ as a \emph{canonical extension} of the family $\big\{\tbeta_k \mid k \ge 0\big\}$. Of course, there are many other interesting possibilities to arrive at suitable extensions. Now the multiplicative extension of the automorphisms
\begin{gather*}
\rho_B(g_k) := \beta_k \quad\text{for}\quad k\ge 0,\qquad 
\rho_M(g_k) := \begin{cases}
 \alpha_0 :=\gamma_0 \beta_0 & \text{for}\ k=0,
 \\
 \alpha_k := \beta_k & \text{for}\ k>0,
 \end{cases}
\end{gather*}
provides us with two representations $\rho_B, \rho_M \colon F \to \Aut(\cM,\psi)$, as it is elementary to verify the relations
\begin{gather}
 \beta_k \beta_\ell = \beta_{\ell+1} \beta_k,
 \qquad\ 0 \le k < \ell < \infty,\nonumber
 \\
 \alpha_k \alpha_\ell = \alpha_{\ell+1} \alpha_k,
 \qquad 0 \le k < \ell < \infty.
 \label{eq:beta-aut}
\end{gather}
Note that \eqref{eq:beta-aut} fails to be valid for $k = \ell$, in contrast to the relations for the partial shifts~$\tbeta_k$
in \eqref{eq:beta-end}.
We have already verified in Proposition~\ref{proposition:MarkovTwoReps2} that
$(\cM, \psi, \alpha_0, \iota(\cA))$ is a bilateral noncommutative Markov process.

The above discussion has provided additional background information on the ideas underlying Theorem \ref{theorem:F-gen-compression}, and on its proof strategy.

\subsection*{Acknowledgements}

The second author was partially supported by a Government of
Ireland Postdoctoral Fellowship (Project ID: GOIPD/2018/498).
Both authors acknowledge several helpful discussions
with B.V.~Rajarama Bhat in an early stage of this project.
Also the first author would like to thank Persi Diaconis, Gwion Evans, Rolf Gohm, Burkhard K\"ummerer
and Hans Maassen for several fruitful discussions on Markovianity.
Both authors thank the organizers of the conference
\emph{Non-commutative algebra, Probability and Analysis in Action}
held at Greifswald in September 2021 in honour of Michael
Sch\"urmann. The authors also thank the anonymous referees for their comments.


\pdfbookmark[1]{References}{ref}
\LastPageEnding

\end{document}